\documentclass[10pt,notitlepage,reqno]{amsart}

\usepackage{verbatim}
\usepackage{amssymb,amsmath,amsthm,mathrsfs,mathtools}
\usepackage{amscd}
\usepackage[breaklinks=true]{hyperref}
\usepackage{url}
\usepackage{units}
\usepackage[usenames,dvipsnames]{xcolor}

\hypersetup{
  colorlinks,
  linkcolor	=	Gray,
  urlcolor 	= Gray,
  citecolor	=	Gray,
  pdfauthor={Mirco Richter},
  pdfkeywords={keywords},
  pdftitle={Hamilton-Lie-Infty-Algebra},
  pdfsubject={n-plectic, symplectic, multisympletic, Poisson, Hamilton, Lie, infinity},
  pdfpagemode=UseNone
}

\let\originalleft\left
\let\originalright\right
\renewcommand{\left}{\mathopen{}\mathclose\bgroup\originalleft}
\renewcommand{\right}{\aftergroup\egroup\originalright}

\newcommand{\bwedge}[1]{\raisebox{0.2ex}{${\textstyle \bigwedge}$}
 \ensuremath{^{\raisebox{-0.2ex}{${\scriptstyle #1}$}}\,}}

\newcommand{\smin}{\,\raisebox{0.06em}{${\scriptstyle \in}$}\,}
\newcommand{\smwedge}{{\scriptstyle \;\wedge\;}}
\newcommand{\ordinal}[1]{[$\;\!#1\,$]}
\newcommand{\R}{\mathbb{R}}
\newcommand{\Z}{\mathbb{Z}}
\newcommand{\N}{\mathbb{N}}
\newcommand{\NN}{\mathbb{N}_0}

\newcommand{\I}{\infty}

\newcommand{\BIGOP}[1]{\mathop{\mathchoice
{\raise-0.22em\hbox{\huge $#1$}}
{\raise-0.05em\hbox{\Large $#1$}}{\hbox{\large $#1$}}{#1}}}
\newcommand{\bigtimes}{\BIGOP{\times}}

\def\clap#1{\hbox to 0pt{\hss#1\hss}}
\def\mathclap{\mathpalette\mathclapinternal}
\def\mathclapinternal#1#2{%
\clap{$\mathsurround=0pt#1{#2}$}%
}

\theoremstyle{plain}
\newtheorem{theorem}{Theorem}[section]
\newtheorem{prop}[theorem]{Proposition}

\newtheorem{corollary}[theorem]{Corollary}
\newtheorem{definition}[theorem]{Definition}
\newtheorem{my_equation}{Equation}

\theoremstyle{remark}
\newtheorem*{remark}{Remark}
\newtheorem{example}{Example}

\title{A Lie infinity-Algebra of Hamiltonian Forms in n-plectic Geometry }
\author{Mirco Richter }
\date {\today}
\thanks{email: \href{mailto:mirco.richter@email.de}{\tt mirco.richter@email.de}}

\begin{document}

\begin{abstract}
We propose a new definition of so called Hamiltonian forms in n-plectic 
geometry and show that they have a non-trivial Lie $\I$-algebra structure.
\end{abstract}
\maketitle
\section{N-Plectic Manifold} Basically \textit{n-plectic} is just another term for
what was once called \textit{multisymplectic} (See \cite{CR}). It generalizes the idea of 
symplectic geometry to manifolds with a distinguished closed and nondegenerate differential form 
of tensor degree higher than two.

A redefinition of the latter was necessary because multisymplectic now referees to
a special kind of vector bundle, designed to have an n-plectic total space 
as well as a valid Darboux theorem.
\begin{definition} For any $n \in \N$, an \textbf{n-plectic manifold} $\left(M,\omega\right)$
is a smooth manifold $M$, together with a differential form $\omega \in \Omega^{n+1}M$,
such that $\omega$ is \textbf{closed} and the map
\begin{equation}
i_{\left(\cdot\right)}\omega_m :T_mM \to \bwedge{n}T^*_mM \; ;\; X \mapsto i_X \omega_m
\end{equation}
is injective for all $m \in M$. In that case we call $\omega$ the \textbf{n-plectic form} of $M$. 
\end{definition}
At a first sight this might look quite similar to the definition of a symplectic manifold, but 
in fact it is a huge generalization. As it turns out for all cases of $n$
intermediate between symplectic forms $(n = 1)$ and volume forms $(n = dim(M)-1)$, there 
is no Darboux theorem without the assumption of additional structure 
\cite{GIM} and according to this $n$-plectic geometry behave in a manner having very 
little in common with symplectic geometry where the Darboux theorem is such a central 
organizing fact \cite{AM}. 
\begin{remark}
In \cite{FO1}, Forger et al. looked in more detail on particular classes of $n$-plectic 
forms and derived conditions under which a generalized Darboux theorem can be expected.
In particular a \textbf{multisymplectic} structure should be a \textbf{fiber bundle} of rank $N$ 
over an $n$-dimensional manifold equipped with a closed and non-degenerate 
$(n+1)$-form $\omega$ defined on the total space $P$ which is $(n-1)$-\textit{horizontal} 
and admits an involutive and isotropic vector subbundle of the vertical bundle $VP$ 
of $P$ of codimension $N$ and dimension $Nn+1$. Under these conditions a generalized Darboux's theorem assures the existence of special 
local coordinates $$x_\mu ,\;\; q^j ,\;\; p_j^\mu ,\;\; p$$ of $P$, called 
\textbf{Darboux coordinates}, in which $\omega$ adopt the form
\begin{equation}\label{multisym_omega}
 \omega~=~dq^j \smwedge dp_j^\mu \smwedge d^{n} x_\mu^{} -
          dp \smwedge d^{n} x.
\end{equation}
In particular the total space of every multisymplectic fiber bundle over some 
n-dimensional base is an $n$-plectic manifold and in this sense '$n$-plectic' is a 
generalization of multisymplectic. 
\end{remark}
Not much is said about morphisms and categories in $n$-plectic geometry so far. 
Even in symplectic geometry there is not a well accepted concept of an
appropriate category yet. A naive but obvious first definition would be the following:
\begin{definition} An \textbf{n-plectic morphism} is a smooth map $f:M\to N$,
such that $(M,\omega_M)$ and $(N,\omega_N)$ are $n$-plectic manifolds and $f$ is
subject to the condition $$\omega_M = f^*\omega_N.$$ If $f$ is a diffeomorphism in
addition, such that the inverse is an $n$-plectic morphism, then $f$ is called an 
\textbf{n-plectomorphism}.
\end{definition}
\section{The Lie $\I$-Algebra of Hamiltonian Forms} 
We define a Lie $\I$-algebra on any $n$-plectic manifold, 
different from the one in \cite{CR}. After a short introduction, 
we suggest a new kind of Hamiltonian 
forms and exhibit their rich and non-trivial
Lie $\I$-algebra structure. We derive explicit expressions
for the bilinear and the trilinear bracket and define the higher
operators inductively.

Since we make extensive use use differential calculus, a short introduction is given in
appendix (\ref{AP_2}). Moreover we have to deal a lot with graded vector spaces and 
sign factors and according to a better readable text we use our own sign
symbols as defined in (\ref{AP_2}).

\subsection{Lie $\I$-Algebras} On the structure level Lie $\I$-algebras generalize 
(differential graded) Lie-algebras to a setting where the Jacobi identity isn't satisfied 
any more, but holds 'up to higher homotopies' only. For more on this topic and how these
algebras are related to the homotopy theory of (co)chain complexes, see for example \cite{LV}. 

A Lie $\I$-algebra can be defined in many different ways \cite{LV}, but the one that
works best for us is its 'graded symmetric, many bracket' version:
\begin{definition} A \textbf{Lie $\I$-algebra} $\left(V,(D_k)_{k\in \N}\right)$
is a graded vector space $V$, together with a sequence $(D_k)_{k\in \N}$ of
graded symmetric, $k$-multilinear morphisms $D_k : \bigtimes^k V \to V$, 
homogeneous of of degree $-1$, such that the \textbf{strong homotopy Jacobi} equation in
dimension $n$
\begin{equation}\label{sh_Jacobi}
\sum_{i+j=n+1} \left(\smashoperator[r]{\sum_{s\in Sh(j,k-j)}}
e\left(s;v_1,\ldots,v_n\right)D_i\left(D_j \left(
v_{s_1}, \ldots, v_{s_j} \right), v_{s_{j+1}}, \ldots, v_{s_n}\right)\right) = 0
\end{equation}
is satisfied for any integer $n \in \N$ and any vectors $v_1,\ldots,v_n \in V$
\footnote{$Sh(p,q)$ is the set of shuffle permutations. See (\ref{AP_2})}. 
\end{definition} 
In particular Lie $\I$-algebras generalizes ordinary Lie algebras, if the grading is chosen 
right:
\begin{example}[Lie Algebra] Every Lie algebra $\left(V,[\cdot,\cdot]\right)$ is a  
Lie $\I$-algebra if we consider $V$ as concentrated in degree one and define $D_k=0$ for any $k \neq 2$ as well as $D_2(\cdot,\cdot):=[\cdot,\cdot]$.
\end{example}
It is beyond the scope of this paper to deal with morphisms or $\I$-morphisms of 
Lie $\I$-algebras. For more about that look for example at \cite{LV}.  
\subsection{Hamiltonian Forms}
We know from symplectic geometry that the non degenerate symplectic 2-form $\omega$ 
gives rise to a well defined pairing 
\begin{equation}\label{fake_fundamental_1}
i_X \omega = df
\end{equation}
between functions $f$ and vector fields $X$ and that this is the origin of the 
Poisson bracket for smooth functions on a symplectic manifold.

In attempt to define something similar in a general n-plectic setting, we 
have to take the following into account: 

For higher n-plectic forms the pairing is capable to exhibit a much 
richer inner structure, since it makes sense for differential forms and multivector 
fields in a range of tensor degrees. However the kernel of $\omega$ is  potentially 
non-trivial on multivector fields of degrees greater than one and consequently the 
association 
\begin{center}
multivector fields $\Leftrightarrow$ differential forms
\end{center}
is not unique in either direction. 

Moreover the associative product and the Jacobi identity of the Poisson bracket depend on 
properties that can't be expected in a general n-plecic setting. If 
the Poisson struture has to be replaced by a more general Lie $\I$-algebra, 
a combination of the strong homotopy Jacobi equation (\ref{sh_Jacobi}) 
and the fundamental pairing (\ref{fake_fundamental_1}) leads to the equation
$$
i_Y\omega = 
-\textstyle\sum^{i>1}_{i+j=k+1} \left(\sum_{s\in Sh(j,k-j)}
D_i\left(D_j\left(f_{s_1}, \ldots, f_{s_j}\right), 
f_{s_{j+1}}, \ldots, f_{s_k}\right)\right)\; .
$$
To ensure the existence of a multivector field solution $Y$, a \textit{second fundamental 
pairing} like
\begin{equation}\label{fake_fundamental_2}
i_Y \omega = f
\end{equation}
is required to hold in addition to the first pairing in (\ref{fake_fundamental_1}).

Fortunately as we will see in (\ref{kernel_prop}) the second pairing then also takes care of the ambiguity inherent in the first pairing due to a possible non trivial kernel of $\omega$.  
Consequently we propose both pairings [(\ref{fake_fundamental_1}) and 
(\ref{fake_fundamental_2})] to have multivector field solutions for a given form.

\begin{remark}Maybe strange at a first sight this is already true in symplectic geometry.
In fact if $\eta$ is the Poisson bivector field associated to the symplectic form 
$\omega$, then the second equation has a solution for any function $f$ since 
$$
i_{\left(f\cdot \eta\right)}\omega =f.
$$ 
\end{remark}
In conclusion, we propose the following new definition of what 
should be called Hamiltonian in $n$-plectic geometry:
\begin{definition} 
A multivector field $X$ on an n-plectic manifold $\left(M,\omega\right)$ is called 
\textbf{semi-Hamiltonian} if there is a differential form $f\in \Omega M$ such that the 
\textbf{first fundamental equation}
\begin{equation}\label{fundamental_1}
i_X \omega=-df
\end{equation}
is satisfied. Conversely, a differential form $f$ on $M$ is
called \textbf{semi-Hamiltonian} if there exists a multivector field such that 
(\ref{fundamental_1}) holds.

In addition a multivector field $Y$ is called \textbf{Hamiltonian} if there exists a 
semi-Hamiltonian form $f$ on $M$ with 
\begin{equation}\label{fundamental_2}
i_Y \omega=-f.
\end{equation}
Conversely, a semi-Hamilton form $f$ is called \textbf{Hamiltonian} if there is a 
multivector field $X$ such that (\ref{fundamental_2}) holds.
\end{definition}
The sign in equation (\ref{fundamental_1}) is necessary for consistency of
our Lie $\I$-algebra structure as we will see later on. The second sign
in equation (\ref{fake_fundamental_2}) is a convention we propose for fancy.

We write $H(M)$ for the set of Hamiltonian forms
on an $n$-plectic manifold $(M,\omega)$. 
Since both fundamental equations are always satisfied by the zero form 
and any element of the kernel of $\omega$ the set $H(M)$ is not empty. 

If a multivector field satisfies one of the 
fundamental equation for a given Hamiltonian form $f$, we say that
it is \textbf{associated} to $f$. We exclusively use designators like $X$ for 
multivector fields satisfying the first equation (i.e. semi-Hamiltonian multivector 
fields) and designators like $Y$ for multivector fields satisfying the second equation 
(i.e. Hamiltonian multivector fields) .
\begin{remark}On $1$-plectic (symplectic) manifolds, this just rephrases the
common definition of Hamiltonian vector fields and functions.  
\end{remark}
On multisymplectic fiber bundles this is equivalent to the definition
of so called Poisson forms. 
\begin{example}
Let $(P\to M,\omega)$ be a multisymplectic fiber bundle and $f \in \Omega(P)$ a 
semi-Hamiltonian form such that $\ker(\omega)\subset \ker(f)$ 
(These forms are called Poisson forms in \cite{FPR}). 
Then there exist a multivector field $Y$ with 
$i_Y\omega = -f$.\footnote{This was shown in \cite{FPR2} only 
for the so called \textit{multiphase space} of a vector bundle, but since the 
proof just requires the local form (\ref{multisym_omega}) of $\omega$ and the existence 
of the Euler vector field, it holds for any multisymplectic fiber bundle.}
\end{example}
In symplectic geometry brackets are defined in terms of associated multivector fields, but
in a general n-plectic setting this association is not necessarily well defined.
To handle the inherent ambiguity, it is required that the 
kernel of $\omega$ is part of the kernel of any Hamiltonian form (as first observed 
in \cite{FPR}). 
We call this the \textbf{kernel property} and as the following
proposition shows it is a consequence of the second pairing:
\begin{prop}\label{kernel_prop}
Assume that $f\in H(M)$ is a Hamiltonian form on an n-plectic manifold 
$\left(M,\omega\right)$. Then
\begin{equation}
\ker(\omega) \subset \ker(f).
\end{equation}
If $Z$ and $Z'$ are semi-Hamiltonian (resp. Hamiltonian) multivector fields associated 
to $f$, then their difference $Z-Z'$ is an element of the kernel of $\omega$ 
and the contractions $i_Zg$ and $i_{Z'}g$ are equal for any Hamiltonian form $g\in H(M)$.
\end{prop}
\begin{proof} The first part is an implication of the second fundamental pairing. Assume $\xi \in \ker(\omega)$. Then there is a Hamiltonian multivector field 
$Y$ with $i_\xi f = i_\xi i_Y\omega =\pm i_Y i_\xi \omega=0$. For the second part compute
$0=f-f=i_Z\omega - i_{Z'}\omega= i_{\left(Z-Z'\right)}\omega$ in case $Z$ and $Z'$
are Hamiltonian as well as 
$0=df-df=i_Z\omega - i_{Z'}\omega= i_{\left(Z-Z'\right)}\omega$ in the semi-Hamiltonian 
situation. Finally we get $i_Zg-i_{Z'}g=i_{(Z-Z')}g=0$ from the kernel property of $g$.
\end{proof}
As we will see in the next sections, this forces our multilinear operators given 
in terms of associated multivector fields to be well defined. The reason is that the contraction of a Hamiltonian form along multivector fields which only 
differ in elements of the kernel of $\omega$ are equal.

Next we examine algebraic structures on Hamiltonian forms. Immediate
from the linearity of the fundamental equations is the following proposition.
\begin{prop}\label{vector_space}
Let $\left(M,\omega\right)$ be an n-plectic manifold. The set of 
Hamiltonian forms is a $\NN$-graded vector space with respect to the tensor grading.
In addition the tensor degree of every (non-trivial) Hamiltonian form 
$f\in H(M)$ is bounded by
$$
0\leq |\,f\,| \leq n+1\;.
$$
\end{prop}
\begin{proof} Since both fundamental equations are linear in their arguments,
$H\left(M\right)$ is a graded vector subspace of $\Omega M$. 

The lower bound is obvious. The upper bound holds, since for a differential
form $f$ of tensor degree $|\,f\,|>(n+1)$ there can't be a non trivial solution
to the second fundamental equation $i_Y\omega=-f$, since $\omega$ is of tensor
degree $n+1$ only.
\end{proof}
Hamiltonian forms of tensor degree $n+1$ have to be closed, since there can't be
a non trivial solution to the first fundamental equation in that case.
In contrast semi-Hamiltonian forms have no upper bound but have to be closed in
tensor degrees greater than $n+1$ for the same reason. 

This is a rather restrictive property with far reaching implications on the 
algebraic structure of Hamiltonian forms.
\begin{corollary}The set $H(M)$ of Hamiltonian forms is not a subalgebra of $\Omega(M)$.
\end{corollary}
\begin{proof}As a consequence of the upper bound on non closed
semi-Hamiltonian forms, the wedge product of semi-Hamiltonian forms is in 
general not a semi-Hamiltonian form.  
\end{proof}
Nevertheless the wedge product still closes on Hamiltonian functions. This is an $n$-plectic generalization of the associative product for functions in symplectic geometry. 
\begin{theorem}\label{ring}
Let $C^\I(M)$ be the algebra of smooth functions on an $n$-plectic
manifold $(M,\omega)$ and $H_0(M)$ the subset of Hamiltonian functions. Then $H_0(M)$ is a
subalgebra of $C^\I(M)$.
\end{theorem}
\begin{proof}
Since the zero function is Hamiltonian, $H_0(M)$ is not empty. By proposition 
(\ref{vector_space}) it is a vector subspace of $C^\I(M)$ and it only remains 
to show that the product of $C^\I(M)$ closes on Hamiltonian functions.

To see that let $f_1,f_2\in H_0(M)$ be Hamiltonian functions, with associated 
semi-Hamiltonian multivector fields $X_1$ and $X_2$ and associated Hamiltonian
multivector fields $Y_1$ and $Y_2$, respectively. 
A semi-Hamiltonian multivector field associated to the product $f_1f_2$
is defined by $f_1X_2 + f_2X_1$ as the computation
$i_{f_1X_2 + f_2X_1}\omega =-f_1df_2-f_2df_1=-d(f_1f_2)$ shows
and an associated Hamiltonian multivector field is given by $f_1Y_2$ (or $f_2Y_1$), since
$i_{f_1Y_2}\omega=-f_1f_2$. 
\end{proof}
An even more important consequence of the upper bound is the following corollary.
It basically tells us that the set of Hamiltonian forms is not always the section 
space of a vector bundle over $M$.
\begin{corollary}\label{negative_module}
In general, the set $H(M)$ of Hamiltonian forms is not a $C^\I(M)$-submodule of $\Omega(M)$.
\end{corollary}
\begin{proof}To see that consider a Hamiltonian form $f$ of
tensor degree $n+1$. An associated Hamiltonian multivector field 
has to be of tensor degree zero (up to elements of the kernel of $\omega$) and hence
is a function $\phi$ satisfying $\phi\omega=-f$.\footnote{Recall that functions are both:
multivector fields as well as differential forms and the contraction is just 
multiplication.}
Since $f$ is closed we have  
$$0=-df=d(\phi\omega)=(d\phi)\wedge\omega+\phi(d\omega)=(d\phi)\wedge \omega$$
and it follows, that $d\phi$ is a \textit{zero divisor} of $\omega$ 
with respect to the wedge product.
Now suppose that there is a function $g\in C^\I(M)$ such that $dg$ is not a zero divisor of
$\omega$. Then $gf$ is not a Hamiltonian form since in general
$$d(gf)= dg\wedge f + gdf = dg \wedge (\phi\omega) = \phi(dg\wedge \omega)\neq 0\;.$$

Hence if there are functions $g\in C^\I(M)$ such that their exterior derivative $dg$ 
is not a zero divisor of $\omega$, then $H(M)$ is not a $C^\I(M)$-submodule of $\Omega(M)$.
\end{proof}
Fortunately it is enough to restrict to zero divisors of the $n$-pletic form to get an
appropriate function ring, that qualifies in a module structure on Hamiltonian forms.
We propose the following definition and proof the module structure later, since it
is easier to use another grading, defined in the next section.
\begin{definition}Let $(M,\omega)$ be an $n$-plectic manifold and $C^\I(M)$ the algebra
of smooth functions on $M$. We say that the set
$$C^\I_\omega(M):=\left\{f\in C^\I(M)\;|\; df\wedge \omega =0 \right\}\;,$$
is the \textbf{n-plectic function algebra} of $M$. Moreover we call a function 
$f\in C^\I_\omega(M)$ an \textbf{$n$-pletic function}. 
\end{definition}
Any constant function is $n$-plectic and so the $C_\omega^\I(M)$ is not empty. 
\begin{prop}
$C^\I_\omega(M)$ is a subalgebra of $C^\I(M)$.
\end{prop}
\begin{proof}From the distributive laws we see that $C_\omega^\I(M)$ is a
vector subspace of $C^\I(M)$. It only remains to show that the product of zero 
divisors is a zero divisor, but this is true, since the wedge product is graded commutative.
\end{proof}
\begin{example}
Let $M$ be an orientable manifold and $\omega$ the volume form on $M$. 
Then $C^\I_\omega(M)$ equals $C^\I(M)$, since $df\wedge\omega=0$ for any 
differential form $f\in \Omega(M)$.
\end{example}
\begin{example}
Let $(P\to M,\omega)$ be an multisymplectic fiber bundle. Then 
$C_\omega^\I(M)$ is the algebra of constant functions on $P$. This was shown in \cite{FPR2}. 
\end{example}
As the following theorem shows, the definition of (semi)-Hamiltonian forms is
at least natural with respect to n-plectomorphisms.
\begin{theorem}
Assume that $\left(M,\omega_M\right)$ and $\left(N,\omega_N\right)$ are n-plectic manifolds and
that $\phi:M\to N$ is an n-plectomorphism. The pullback $\phi^*f$ is a Hamiltonian form on
$M$ for any Hamiltonian form $f\in H(N)$ on $N$.  
\end{theorem}
\begin{proof}
Since $\phi$ is a diffeomorphism we can pull any multivector
field $X$ on $N$ back to a multivector field $\phi^*X$ on $M$. 

It remains to show that if $X$ is associated to $f$ then
$\phi^*X$ is associated to $\phi^*f$, but that follows, since the exterior derivative is 
natural and the contraction has the natural property (\ref{multi_rules}). 
In particular we have 
$i_{X'}\omega_M=i_{X'}\phi^*\omega_N=\phi^*i_X\omega_N=-\phi^*df=-d\phi^*f$ 
for the first pairing and a  similar calculation for the second.
\end{proof}
\subsection{The Differential} We define a differential for Hamiltonian forms 
and give another grading, slightly different from the usual tensor grading. The
latter will simplify our calculations in the 
following sections and leads to a \textit{symmetric} Lie $\I$-algebra.
\begin{definition}[Symmetric Grading] 
Let $\left(M,\omega\right)$ be an n-plectic manifold and $f\in H(M)$ a Hamiltonian form 
homogeneous of tensor degree $r$. The \textbf{symmetric degree} of $f$ is  
\begin{equation}
deg(f):=n-r \; .
\end{equation}
To distinguish it from the usual tensor grading we use the symbol $|\,\cdot\,|$ 
exclusively for the tensor degree.
\end{definition}
An immediate consequence of the fundamental equations is that the 
tensor degree of any Hamiltonian multivector field $Y$ associated to a homogeneous 
Hamiltonian form $f$ is given by
\begin{equation}
|\,Y\,|=deg(f)+1.
\end{equation}
If $f$ is not closed, the tensor degree of an associated semi-Hamiltonian multivector field 
$X$ is
$$|\,X\,|=deg(f).$$
\begin{prop}Let $\left(M,\omega\right)$ be an n-plectic manifold. The set of 
Hamiltonian forms is a $\Z$-graded vector space with respect to the symmetric grading.
If $f\in H(M)$ is a Hamiltonian form its symmetric degree is bounded by
$$
-1\leq deg(f) \leq n.
$$
\end{prop}
\begin{proof} Follows from proposition (\ref{vector_space}).
\end{proof}
We choose this grading to simplify our calculations. 
If one wants Hamiltonian forms to be bounded by zero, the grading can be shifted. 
In that case we get a Lie $\I$-algebra in its graded \textit{skew}-symmetric incarnation.
Unless otherwise stated, we assume that Hamiltonian forms are graded with respect to the 
symmetric grading.

Besides the vector space structure, Hamiltonian forms have an additional module structure
related to the ring of $n$-plectic functions.
\begin{theorem}
Let $(M,\omega)$ be an $n$-plectic manifold and 
$C_\omega^\I(M)$ the $n$-plectic function algebra. The set of Hamiltonian forms $H(M)$ is a $C_\omega^\I(M)$-module.

If $f \in C_\omega^\I(M)$ is an $n$-plectic function and $g$ a homogeneous Hamiltonian 
form with associated semi-Hamiltonian and Hamiltonian multivector field $X$ and $Y$, 
respectively, then an associated Hamiltonian multivector field of the scalar product 
$fg$ is given by
\begin{equation}
f Y
\end{equation}
and an associated semi-Hamiltonian multivector field of the scalar product $fg$ is 
given by 
\begin{equation}
e(f)[f,Y]+fX.
\end{equation}
\end{theorem}
\begin{proof}The first part follows from the second, since we only have to show
that the $C_\omega^\I(M)$ scalar multiplication closes on Hamiltonian forms. 
This can be seen by the direct calculation
$$i_{f\cdot Y}\omega=-fg$$ 
in case of the associated Hamiltonian multivector field and
\begin{align*}
i_{-e(f)[f,Y]+ fX}\omega
&= -e(f)i_{[f,Y]}\omega -fdg\\ 
&= L_{f}i_Y\omega-i_YL_f\omega -fdg\\
&= -L_fg -i_Yd(f\omega) -fdg\\
&= -d(fg)+fdg-i_Ydf\wedge\omega-fdg\\
&= -df\wedge g - fdg\\
&= D_1(fg)
\end{align*}
in case of the associated semi-Hamiltonian multivector field.
\end{proof}

Back on topic, the following proposition qualifies the exterior derivative as a valid differential on Hamilton forms. 
\begin{prop}\label{sym_differential}Let $\left(M,\omega\right)$ be an $n$-plectic
manifold. With respect to the symmetric grading, the exterior derivative $d$ is a 
differential on $H\left(M\right)$.
\end{prop}
\begin{proof} Since the exterior derivative is a codifferential on differential forms
with respect to the tensor grading it only remains to show that it is homogeneous of
symmetric degree $-1$ and closes on Hamiltonian forms. 

The former can be seen from $deg(df)=n-(|f|+1)= (n-|f|)-1=deg(f)-1$. For the latter, 
assume $f\in H(M)$ with associated semi-Hamiltonian multivector field 
$X$. Then $df$ is closed and an associated semi-Hamiltonian multivector is given by 
any element of the kernel of $\omega$, while an associated Hamiltonian multivector field 
is given by $X$.
\end{proof}
Regarding our proposed Lie $\I$-algebra on Hamiltonian forms we will use the negative
exterior derivative as the differential for consistency reasons.
\begin{definition}
Let $(M, \omega)$ be an n-plectic manifold and $H(M)$ the $\Z$-graded
vector space of Hamiltonian forms on $M$. The \textbf{$n$-plectic differential}
\begin{equation}
D_1 : H(M) \to H(M)
\end{equation}
of $H(M)$ is defined for any Hamiltonian form $f\in H(M)$ by the negative exterior derivative
$$
D_1f:=-df\;.
$$
\end{definition}
In terms of the $n$-plectic differential, the first fundamental equation then just reads 
as $i_X\omega = D_1f$.
\subsection{The Bilinear Operator}\label{subsec_binary} We generalize 
the usual Poisson bracket of functions on a symplectic manifold (see 
for example \cite{AM}) to Hamiltonian forms in a general n-plectic framework. Inspired 
by the Poisson bracket in \cite{FPR}, our bilinear operator is just a 
(graded symmetric) adoption of the former to a differential graded 
setting where no Hamiltonian primitive of the $n$-plectic form is available.  

\begin{definition}\label{D_2}
Let $\left(M, \omega\right)$ be an n-plectic manifold and $H\left(M\right)$ the $\Z$-graded vector space of Hamiltonian forms on $M$. The 
\textbf{strong homotopy Lie 2-bracket}
\begin{equation}
D_2:  H(M)\times H(M) \to H(M) \\
\end{equation}
is defined for any homogeneous $f_1$, $f_2\in H\left(M\right)$ and associated 
semi-Hamiltonian multivector fields $X_1$ and $X_2$ by 
$$D_2\left(f_1,f_2\right):= e\left(f_1\right)L_{X_{1}}f_2 
	+e\left(f_1,f_2\right)e\left(f_2\right)L_{X_{2}}f_1$$  
and is then extended to $H\left(M\right)$ by linearity. 
\end{definition}
Also similar operators (\cite{FPR},\cite{AM}) are usually called 
\textit{Poisson} bracket, this is misleading in our
context, since there is no known product to define a Poisson ($\I$)-algebra 
on $H(M)$ in general. 
In addition we propose the '\textit{strong homotopy}' modifier since theorem
(\ref{jacobi_id}) shows that the Jacobi identity does not vanish but holds 'up to higher homotopies' as we will see in the next section. 

On the technical level, a first thing to show is, that the bracket is 
independent of the particular chosen associated semi-Hamiltonian multivector fields. This is guaranteed by proposition (\ref{kernel_prop}). 
\begin{theorem}\label{well_def_2}
For any two $f_1$, $f_2\in H\left(M\right)$, the image 
$D_2(f_1,f_2)$ is a well defined Hamiltonian form. If $Y_1$ resp. $Y_2$ are
associated Hamiltonian multivector fields and $X_1$ resp. $X_2$ are associated
semi-Hamiltonian multivector fields, then an associated Hamiltonian 
multivector field $Y_{D_2}\left(f_1,f_2\right)$ is given by
\begin{equation}
\left[Y_2,X_1\right]
  +e\left(f_1,f_2\right)\left[Y_1,X_2\right]
\end{equation}
and an associated semi-Hamiltonian multivector field $X_{D_2}\left(f_1,f_2\right)$ by
\begin{equation}
-2e\left(f_1\right)\left[X_2,X_1\right]\; .
\end{equation}
\end{theorem}
\begin{proof}
To see that $D_2$ is well defined, suppose $\xi$ is a multivector fields from the 
kernel of $\omega$. Then $L_{\xi}f=di_{\xi}f-(-1)^{|\,\xi\,|}i_{\xi}df=0$ since $f$ as
well as $df$ has the kernel property and we get 
$L_{X+\xi}f=L_{X}f$ for any Hamiltonian form $f$ and semi-Hamiltonian multivector field $X$. 

By prop. (\ref{kernel_prop}) the difference of multivector fields associated to 
the same Hamiltonian form is an element of the kernel of $\omega$ and consequently the image 
$D_2(f_1,f_2)$ does not depend on the particular chosen associated semi-Hamiltonian multivector 
field.

To see that $[Y_2,X_1]+e(f_1,f_2)[Y_1,X_2]$ is an associated Hamiltonian multivector 
field, compute
\begin{align*}
i_{Y_{D_2}\left(f_1,f_2\right)}\omega
&=i_{\left[Y_2,X_1\right]}\omega +e\left(f_1,f_2\right)i_{\left[Y_1,X_2\right]}\omega\\
&=- e\left(f_1,f_2\right)e\left(f_2\right)i_{\left[X_1,Y_2\right]}\omega
	-e\left(f_1\right)i_{\left[X_2,Y_1\right]}\omega\\
&=-e\left(f_1,f_2\right)e\left(f_2\right)\left(
	-e\left(f_1,f_2\right)e\left(f_1\right)e\left(f_2\right)
		L_{X_1}i_{Y_2}\omega-i_{Y_2}L_{X_1}\omega\right)\\
&\quad-e\left(f_1\right)\left(-e\left(f_1,f_2\right)e\left(f_1\right)e\left(f_2\right)
		L_{X_2} i_{Y_1}\omega-i_{Y_1}L_{X_2}\omega\right)\\	
&= e\left(f_1\right)L_{X_1}i_{Y_2}\omega + 
	e\left(f_1,f_2\right)e\left(f_2\right)L_{X_2} i_{Y_1}\omega\\
&= -e\left(f_1\right)L_{X_1}f_2 - e\left(f_1,f_2\right)e\left(f_2\right)
		L_{X_2} f_1\\
&= -D_2\left(f_1,f_2\right)
\end{align*}
and to see that $-2e(f_1)[X_2,X_1]$ is an associated semi-Hamilton multivector field 
use $L_X\omega=0$ (for a semi-Hamiltonian multivector field) and compute 
\begin{align*}
-2e\left(f_1\right)i_{[X_2,X_1]}\omega
&=-e\left(f_1\right)i_{[X_2,X_1]} \omega 
	-e\left(f_1,f_2\right)e\left(f_2\right)i_{[X_1,X_2]} \omega\\
&=-e\left(f_1,f_2\right)L_{X_2} i_{X_1} \omega
   +e\left(f_1\right)i_{X_1} L_{X_2} \omega  \\
&\quad-L_{X_1} i_{X_2} \omega
	+ e\left(f_1,f_2\right)e\left(f_2\right)i_{X_2} L_{X_1} \omega \\
&=- e\left(f_1,f_2\right)e\left(f_2\right)e\left(f_2\right)
		L_{X_2} D_1f_1- e\left(f_1\right)e\left(f_1\right)L_{X_1} D_1f_2\\
&= e\left(f_1,f_2\right)e\left(f_2\right)e\left(f_2\right)
		L_{X_2} df_1+ e\left(f_1\right)e\left(f_1\right)L_{X_1}df_2\\
&= -e\left(f_1,f_2\right)e\left(f_2\right)
		dL_{X_2}f_1- e\left(f_1\right)dL_{X_1}f_2\\
&= e\left(f_1,f_2\right)e\left(f_2\right)
		D_1L_{X_2}f_1+ e\left(f_1\right)D_1L_{X_1}f_2\\
&=D_1D_2\left(f_1,f_2\right)\;.
\end{align*}
\end{proof}
In what follows we will sometimes referee to $X_{D_2}$ as a semi-Hamiltonian 
multivector field, associated to $D_2$, without stating the arguments explicit.
\begin{corollary}
The $\NN$-graded vector space of semi-Hamiltonian multivector fields is a subalgebra
of the Schouten algebra.
\end{corollary}
\begin{proof}Semi-Hamiltonian multivector fields associated to $(n+1)$-forms are trivial, 
so the grading is valid. Since the first fundamental equation is linear it only remains 
to show that the Schouten bracket closes on semi-Hamiltonian multivector fields, but this is 
guaranteed by the previous theorem.
\end{proof}
\begin{remark}
The associated multivector field $Y_{D_2}\left(f_1,f_2\right)$ is graded symmetric 
in its arguments and from the super-symmetry of the Schouten bracket we get
$$-2e\left(f_1\right)\left[X_2,X_1\right]=
	-2e\left(f_1,f_2\right)e\left(f_2\right)\left[X_1,X_2\right]\;,$$ so that
$X_{D_2}(f_1,f_2)$ is graded symmetric too. 
If at least one of the arguments is a closed form, $X_{D_2}(f_1,f_2)$ 
vanishes. Consequently closed forms are a two-sided ideal in the non-associative algebra 
$\left(H(M), D_2\right)$.
\end{remark}
The next theorem shows that the strong homotopy Lie $2$-bracket qualifies as the bilinear
operator in a Lie $\I$-algebra, in the sense that it has the right symmetry and interacts 
with the differential as required.

\begin{theorem}
The bilinear operator $D_2$ is graded symmetric and homogeneous of degree $-1$ with respect to the symmetric grading. Moreover the strong homotopy Jacobi equation
in dimension two
\begin{equation}
D_1D_2\left(f_1,f_2\right)+D_2\left(D_1f_1,f_2\right)+
	e\left(f_1,f_2\right)D_2\left(D_1f_2,f_1\right)=0
\end{equation}
is satisfied for any Hamiltonian forms $f_1,f_2 \in H(M)$.
\end{theorem}
\begin{proof}
Assume $X_1$, $X_2 \in \mathfrak{X}M$ are semi-Hamiltonian multivector fields, 
associated to $f_1$ and $f_2$, respectively. 

Bilinearity is a straight forward implication of the definition, since a multivector 
field associated to any linear combination $\lambda_1f_1+\lambda_2f_2$ is given by 
the linear combination $\lambda_1X_1+\lambda_2X_2$. 
To see graded symmetry, compute
\begin{align*}
D_2\left(f_1,f_2\right) &= e\left(f_1\right)L_{X_1}f_2 
	+e\left(f_1,f_2\right)e\left(f_2\right)L_{X_2}f_1\\
&= e\left(f_1,f_2\right)\left(e\left(f_1,f_2\right)e\left(f_1\right)
	L_{X_1}f_2 + e\left(f_2\right)L_{X_2}f_1\right)\\
&= e\left(f_1,f_2\right)\left(e\left(f_2\right)
	L_{X_2}f_1+e\left(f_2,f_1\right)e\left(f_1\right)L_{X_1}f_2\right)\\
&= e\left(f_1,f_2\right)D_2\left(f_2,f_1\right).
\end{align*}
If $f_i$ is homogeneous of symmetric degree $deg(f_i)$ the Lie derivative 
along $X_i$ is homogeneous of degree $deg(f_i)-1$. Consequently 
$deg(D_2(f_1,f_2))=deg(f_1)+deg(f_2)-1$ and $D_2$ is homogeneous of 
symmetric degree $-1$.

Finally, compute the strong homotopy Jacobi equation in dimension two:
\begin{align*}
D_1D_2\left(f_1, f_2\right) &= -e\left(f_1\right)dL_{X_1}f_2
	-e\left(f_1,f_2\right)e\left(f_2\right)dL_{X_2}f_1\\
&= e\left(f_1\right)e\left(f_1\right) 
	L_{X_1}df_2 +e\left(f_1,f_2\right)e\left(f_2\right)e\left(f_2\right)L_{X_2}df_1\\	
&=-e\left(f_1\right)e\left(f_1\right)L_{X_1}\left(-df_2\right) 
	-e\left(D_1f_1,f_2\right)e\left(f_2\right)L_{X_2}\left(-df_1\right)\\ 	
&= -D_2\left(D_1f_1, f_2\right) - e\left(f_1\right)D_2\left(f_1, D_1f_2\right)\\	
&= -D_2\left(D_1f_1,f_2\right) 
	- e\left(f_1,df_2\right)e\left(f_1\right)D_2\left(D_1f_2, f_1\right)\\	
&= -D_2\left(D_1f_1, f_2\right) - e\left(f_1,f_2\right) D_2\left(D_1f_2, f_1\right).
\end{align*}
\end{proof}
In case $\omega$ is symplectic, (non-closed) Hamiltonian forms are just functions and definition
(\ref{D_2}) rephrases the usual Poisson bracket from symplectic geometry to a
differential graded setting.
However a big difference is, that in general the Jacobi identity does not vanish any more. 
\begin{theorem}[Jacobi Identity]\label{jacobi_id} 
The graded Jacobi identity does not vanish. 
Instead for any three Hamiltonian forms 
$f_1$, $f_2$, $f_3 \in H(M)$ the equation
\begin{multline*}
\smashoperator[r]{\sum_{s\in Sh\left(2,1\right)}}e\left(s;f_1,f_2,f_3\right)
	D_2\left(D_2\left(f_{s_1},f_{s_2}\right),f_{s_3}\right)\\
= -\tfrac{1}{2}\sum_{\mathclap{s\in Sh\left(2,1\right)}}e\left(s;f_1,f_2,f_3\right)
		e\left(f_{s_1}\right)e\left(f_{s_2}\right)
			L_{X_{D_2}\left(f_{s_1},f_{s_2}\right)}f_{s_3}
\end{multline*}
is satisfied.
\end{theorem} 
\begin{proof} Let $X_1$, $X_2$, $X_3$ be associated semi-Hamiltonian multivector fields, 
respectively. We apply the definition of $D_2$ to
rewrite the left side into
\begin{multline*}
-\textstyle\sum_{s\in Sh(2,1)}e\left(s;f_1,f_2,f_3\right)
	e\left(f_{s_1}\right)e\left(f_{s_2}\right)
		L_{X_{D_2}\left(f_{s_1},f_{s_2}\right)}f_{s_3}\\
+\textstyle\sum_{s\in Sh(2,1)}e\left(s;f_1,f_2,f_3\right)
	e\left(f_{s_1},f_{s_3}\right)e\left(f_{s_2},f_{s_3}\right)L_{X_{s_3}}\\
\cdot\left(e\left(f_{s_1}\right)L_{X_{s_1}}f_{s_2} 
	+ e\left(f_{s_1},f_{s_2}\right)e\left(f_{s_2}\right)
		L_{X_{s_2}}f_{s_1}\right)=
\end{multline*}
\begin{multline*}
-\textstyle\sum_{s\in Sh(2,1)}e\left(s;f_1,f_2,f_3\right)
	e\left(f_{s_1}\right)e\left(f_{s_2}\right)
		L_{X_{D_2}\left(f_{s_1},f_{s_2}\right)}f_{s_3}\\
+\textstyle\sum_{s\in Sh(2,1)}e\left(s;f_1,f_2,f_3\right)
	e\left(f_{s_2}\right)L_{X_{s_1}}L_{X_{s_2}}f_{s_3}\\
+\textstyle\sum_{s\in Sh(2,1)}e\left(s;f_1,f_2,f_3\right)
	e\left(f_{s_1}\right)e\left(f_{s_1},f_{s_2}\right)L_{X_{s_2}}L_{X_{s_1}}f_{s_3}=	
\end{multline*}
\begin{multline*}
-\textstyle\sum_{s\in Sh(2,1)}c_2e\left(s;f_1,f_2,f_3\right)
	e\left(f_{s_1}\right)e\left(f_{s_2}\right)
		L_{X_{D_2}\left(f_{s_1},f_{s_2}\right)}f_{s_3}\\
-\textstyle\sum_{s\in Sh(2,1)}e\left(s;f_1,f_2,f_3\right)
	e\left(f_{s_2}\right)\\
		\cdot\left(-e\left(f_{s_1},f_{s_2}\right)e\left(f_{s_1}\right)
		e\left(f_{s_2}\right)L_{X_2}L_{X_{s_1}}f_{s_3}
			-L_{X_{s_1}}L_{X_{s_2}}f_{s_3}\right)\;.
\end{multline*}
Using (\ref{multi_rules}) the second shuffle sum can be rewritten in terms of 
the Schouten bracket to get the expression
\begin{multline*}
-\textstyle\sum_{s\in Sh(2,1)}e\left(s;f_1,f_2,f_3\right)
	e\left(f_{s_1}\right)e\left(f_{s_2}\right)
		L_{X_{D_2}\left(f_{s_1},f_{s_2}\right)}f_{s_3}\\
-\textstyle\sum_{s\in Sh\left(2,1\right)}e\left(s;f_1,f_2,f_3\right)
	e\left(f_{s_2}\right)L_{[X_{s_2},X_{s_1}]}f_{s_3}
\end{multline*}
and expanding this 
\begin{multline*}
-\textstyle\sum_{s\in Sh\left(2,1\right)}e\left(s;f_1,f_2,f_3\right)
	e\left(f_{s_1}\right)e\left(f_{s_2}\right)
		L_{X_{D_2}\left(f_{s_1},f_{s_2}\right)}f_{s_3}\\
+\textstyle\frac{1}{2}\sum_{s\in Sh\left(2,1\right)}
	e\left(s;f_1,f_2,f_3\right)
		e\left(f_{s_1}\right)e\left(f_{s_2}\right)
			L_{-2c_2e(f_{s_1})[X_{s_2},X_{s_1}]}f_{s_3}
\end{multline*}
we applay the definition of the associated semi-Hamiltonian multivector field 
and collect the shuffle terms to arrive at
\begin{align*}		
-\textstyle\frac{1}{2}\sum_{s\in Sh\left(2,1\right)}
	c_2e\left(s;f_1,f_2,f_3\right)	e\left(f_{s_1}\right)e\left(f_{s_2}\right)
		L_{X_{D_2}\left(f_{s_1},f_{s_2}\right)}f_{s_3}\;.
\end{align*}
\end{proof}
\begin{remark}
The factor 'two' in  $-2e(f_1)[X_2,X_1]$ is the only reason 
for the graded Jacobi expression to not vanish. In addition the Jacobi equation
as given above, is still valid for semi-Hamiltonian forms.
\end{remark}
The next proposition uses the vanishing Jacobi identity of the Schouten bracket to show, that the Hamiltonian Jacobi expression is a closed form.

\begin{theorem}\label{J2_mvf}
For any three Hamiltonian forms $f_1$, $f_2$, $f_3\in H(M)$, the Jacobi expression
$$\smashoperator[r]{\sum_{s\in Sh(2,1)}}e\left(s;f_1,f_2,f_3\right)
D_2\left(D_2\left(f_{s_1},f_{s_2}\right),f_{s_3}\right)$$
is a closed Hamiltonian form. If $X_1$, $X_2$, $X_3$ are associated semi-Hamiltonian
multivector fields and $Y_1$, $Y_2$, $Y_3$ are associated Hamiltonian multivetor field, 
respectively, then an associated Hamiltonian multivector field is 
given by
\begin{equation}\label{J_2}
\tfrac{1}{2}\sum_{\mathclap{s\in Sh(2,1)}}e\left(s;f_1,f_2,f_3\right)
e\left(f_{s_2}\right)\left[\left[X_{s_3},X_{s_2}\right],Y_{s_1}\right]
\end{equation}
\end{theorem}
\begin{proof}
To see that (\ref{J_2}) is an associated Hamiltonian multivector field compute
\begin{align*}
&\textstyle\frac{1}{2}\sum_{s\in Sh\left(2,1\right)}e\left(s;f_1,f_2,f_3\right)
	e\left(f_{s_2}\right)i_{[[X_{s_3},X_{s_2}],Y_{s_1}]}\omega\\
&=-\textstyle\frac{1}{2}\sum_{s\in Sh\left(2,1\right)}e\left(s;f_1,f_2,f_3\right)
	e\left(f_{s_2}\right)e\left(f_{s_1},f_{s_2}\right)e\left(f_{s_1},f_{s_3}\right)
			e\left(f_{s_2}\right)e\left(f_{s_3}\right)L_{[X_{s_3},X_{s_2}]}i_{Y_{s_1}}\omega\\
&=\textstyle\frac{1}{2}\sum_{s\in Sh\left(2,1\right)}e\left(s;f_1,f_2,f_3\right)
	e\left(f_{s_1}\right)e\left(f_{s_2}\right)
		L_{-2e\left(f_{s_1}\right)[X_{s_2},X_{s_1}]}i_{Y_{s_3}}\omega\\
&=\textstyle\frac{1}{2}\sum_{s\in Sh\left(2,1\right)}e\left(s;f_1,f_2,f_3\right)
	e\left(f_{s_1}\right)e\left(f_{s_2}\right)
		L_{X_{D_2}\left(f_{s_1},f_{s_2}\right)}i_{Y_{s_3}}\omega\\
&=\textstyle\frac{1}{2}\sum_{s\in Sh\left(2,1\right)}e\left(s;f_1,f_2,f_3\right)
	e\left(f_{s_1}\right)e\left(f_{s_2}\right)
		L_{X_{D_2}\left(f_{s_1},f_{s_2}\right)}i_{Y_{s_3}}\omega\\
&=\textstyle\frac{1}{2}\sum_{s\in Sh\left(2,1\right)}e\left(s;f_1,f_2,f_3\right)
	e\left(f_{s_1}\right)e\left(f_{s_2}\right)
		L_{X_{D_2}\left(f_{s_1},f_{s_2}\right)}f_{s_3}\\
&=-\textstyle\sum_{s\in Sh\left(2,1\right)}e\left(s;f_1,f_2,f_3\right)
	D_2\left(D_2\left(f_{s_1},f_{s_2}\right),f_{s_3}\right)\; .
\end{align*}
The closedness follows from the Jacobi identity of the Schouten bracket for 
multivector fields. Since $[X_i,X_j]$ is a semi-Hamiltonian multivector field, $L_{[X_i,X_j]}\omega=0$ and we have
\begin{align*}
0&=\textstyle\sum_{s\in Sh(2,1)}e\left(s;f_1,f_2,f_3\right)
		e\left(f_{s_1}\right)e\left(f_{s_3}\right)
			i_{[[X_{s_3},X_{s_2}],X_{s_1}]}\omega\\
&=\textstyle\sum_{s\in Sh(2,1)}e\left(s;f_1,f_2,f_3\right)
		e\left(f_{s_1},f_{s_2}\right)e\left(f_{s_1},f_{s_3}\right)
			e\left(f_{s_1}\right)e\left(f_{s_3}\right)
				L_{[X_{s_3},X_{s_2}]}i_{X_{s_1}}\omega\\
&=\textstyle\sum_{s\in Sh(2,1)}e\left(s;f_1,f_2,f_3\right)
	e\left(f_{s_2}\right)e\left(f_{s_3}\right)
		L_{[X_{s_2},X_{s_1}]}i_{X_{s_3}}\omega\\
&=-\textstyle\frac{1}{2}\sum_{s\in Sh(2,1)}e\left(s;f_1,f_2,f_3\right)
	e\left(f_{s_1}\right)e\left(f_{s_2}\right)e\left(f_{s_3}\right)
		L_{-2e\left(f_{s_1}\right)[X_{s_2},X_{s_1}]}df_{s_3}\\
&=-\textstyle\frac{1}{2}e\left(f_1\right)e\left(f_2\right)e\left(f_3\right)
	\sum_{s\in Sh(2,1)}e\left(s;f_1,f_2,f_3\right)
		L_{X_{D_2}\left(f_{s_1},f_{s_2}\right)}df_{s_3}\\
&=-\textstyle\frac{1}{2}e\left(f_1\right)e\left(f_2\right)e\left(f_3\right)
	\sum_{s\in Sh(2,1)}e\left(s;f_1,f_2,f_3\right)
	e\left(f_{s_1}\right)e\left(f_{s_2}\right)
		dL_{X_{D_2}\left(f_{s_1},f_{s_2}\right)}f_{s_3}\\
&=\textstyle e\left(f_1\right)e\left(f_2\right)e\left(f_3\right)
	d\left(\sum_{s\in Sh(2,1)}e\left(s;f_1,f_2,f_3\right)
		D_2\left(D_2\left(f_{s_1},f_{s_2}\right),f_{s_3}\right)\right)\;.
\end{align*}
\end{proof}
The existence of the associated multivector field (\ref{J2_mvf})
is tied to the assumption that for any Hamiltonian form $f$ there is
a multivector field $Y$ satisfying $i_Y\omega=-f$. If we drop that assumption and
work with semi-Hamiltonian forms instead, the equation
\begin{equation}\label{help_1}
i_Y\omega = -\textstyle\sum_{s\in Sh(2,1)}e\left(s;f_1,f_2,f_3\right)
D_2\left(D_2\left(f_{s_1},f_{s_2}\right),f_{s_3}\right)
\end{equation}
must not have a solution any more, as the following simple counterexample shows:
\begin{example}\label{counterexample_1}
Let $M$ be the linear manifold $\R^6$ with global coordinates $x_1,\ldots,x_6$
and $\omega$ the differential $4$-form
$$ 
dx^1\smwedge dx^3\smwedge dx^5 \smwedge dx^6 +
dx^2\smwedge dx^4\smwedge dx^5 \smwedge dx^6\;.
$$
$\left(\R^6,\omega\right)$ is a $3$-plectic manifold, 
since $\omega$ is closed and the contraction $i_{X^j\partial_j}\omega$ is the zero
form, only if any coordinate $X^j$ is the zero function.
Define
$$
\begin{array}{lcr}
f_1:= \left(x_4-x_1^2\,x_3\right)\,dx^5\smwedge dx^6 
&\mbox{ and }&
f_2:= \left(x_3+x_2^2\,x_4\right)\,dx^5\smwedge dx^6\;.
\end{array}
$$
These forms are Hamiltonian, since associated Hamiltonian multivector fields are given 
(for example) by
$$
\begin{array}{lcr}
Y_1:= \left(x_1^2\,x_3-x_4\right)\,\partial_3\smwedge \partial_1 
&,&
Y_2:= -\left(x_2^2\,x_4+x_3\right)\,\partial_4\smwedge \partial_2
\end{array}
$$
and associated semi-Hamiltonian multivector fields by
$$
\begin{array}{lcr}
X_1:= x_1^2\partial_1 -\partial_2 -2x_1x_3\partial_3 
&,&
X_2:= -\partial_1 -x_2^2\partial_2 +2x_2x_4\partial_4 \;.
\end{array}
$$
In this case the Schouten bracket reduces to the usual Lie bracket and is given by
$$
[X_2,X_1]=2x_1\partial_1+2x_2\partial_2-2x_3\partial_3-2x_4\partial_4\;.
$$
Next define the form
$$
f_3:= dx^1 \smwedge dx^2\; .
$$
Any closed form is semi-Hamiltonian and so is $f_3$. In contrast $f_3$ is 
not Hamiltonian, because there can't be a multivector field satisfying
$i_Y\omega=-f_3$. 

To compute the Jacobi expression $\textstyle\sum_{s\in Sh(2,1)}e\left(s;f_1,f_2,f_3\right)
D_2\left(D_2\left(f_{s_1},f_{s_2}\right),f_{s_3}\right)$ we use the fact that equation
(\ref{jacobi_id}) is still valid for semi-Hamiltonian forms. Since $f_3$ is closed we have to compute $L_{[X_2,X_1]}f_3$ only, but this is
$4 dx^1 \smwedge dx^2$.
It follows that the Jacobi expression is not a Hamiltonian form and that equation 
(\ref{help_1}) has no solution.
\end{example} 

This justifies our proposed definition of Hamiltonian forms. If we want a trilinear 
operator $D_3$, related to our sh-Lie $2$-bracket by the strong homotopy Jacobi
equation in dimension three a solution to the equation
\begin{multline*}
i_{X_{D_3}\left(f_1,f_2,f_3\right)}\omega =
-\textstyle\sum_{s\in Sh(2,1)}e\left(s;f_1,f_2,f_3\right)
D_2\left(D_2\left(f_{s_1},f_{s_2}\right),f_{s_3}\right)\\
-\textstyle\sum_{s\in Sh(1,2)}e\left(s;f_1,f_2,f_3\right)
D_3\left(D_1f_{s_1},f_{s_2},f_{s_3}\right)
\end{multline*}
is required and since the contraction operator is linear, this needs 
a solution of equation (\ref{help_1}).
\subsection{The Trilinear Operator}\label{subsec_trinary} The strong homotopy Lie $2$-bracket 
does not satisfy the graded Jacobi identity and it is the subject of this section to 
define a trilinear operator $D_3$ such that the \textit{strong homotopy} Jacobi equation 
(\ref{sh_Jacobi}) in dimension three is satisfied instead.  
\begin{definition}Let $\left(M, \omega\right)$ be an n-plectic manifold and $H\left(M\right)$ the $\Z$-graded vector space of Hamiltonian forms on $M$. The 
\textbf{strong homotopy Lie 3-bracket}
\begin{equation}
D_3:  H(M)\times H(M)\times H(M) \to H(M) \\
\end{equation}
is defined for any homogeneous $f_1,f_2,f_3\in H\left(M\right)$ and semi-Hamiltonian multivector 
fields $X_{D_2}\left(\cdot,\cdot\right)$ associated to the sh-Lie $2$-bracket, by 
$$
D_3(f_1,f_2,f_3):= 
	-\tfrac{1}{2}\sum_{\mathclap{s\in Sh(2,1)}}e\left(s;f_1,f_2,f_3\right)
	e\left(f_{s_1}\right)e\left(f_{s_2}\right)
		i_{X_{D_2}\left(f_{s_1},f_{s_2}\right)}f_{s_3}
$$ 
and is then extended to $H\left(M\right)$ by linearity. 
\end{definition}

Again the kernel property (\ref{kernel_prop}) guarantees that this definition does not
depend on the particular chosen associated semi-Hamiltonian multivector fields. Moreover
it qualifies as the trilinear operator in a Lie $\I$-algebra structure:

\begin{theorem}
The operator $D_3$ is well defined, graded symmetric and homogeneous of degree $-1$. For 
any three Hamiltonian forms $f_1$, $f_2$, $f_3 \in H(M)$ the strong homotopy Jacobi equation
in dimension four
\begin{multline}
D_1D_3\left(f_1,f_2,f_3\right)+\sum_{\mathclap{s\in Sh(1,2)}}e\left(s;f_1,f_2,f_3\right)
	D_3\left(D_1f_{s_1},f_{s_2},f_{s_3}\right)\\
		+\sum_{\mathclap{s\in Sh(2,1)}}e\left(s;f_1,f_2,f_3\right)
			D_2(D_2(f_{s_1},f_{s_2}),f_{s_3})=0
\end{multline}
is satisfied and for associated Hamiltonian multivector fields 
$Y_1$, $Y_2$, $Y_3$  an associated Hamiltonian multivector field 
$Y_{D_3}\left(f_1,f_2,f_3\right)$ is given by
\begin{equation}\label{ass_Ham_3}
\tfrac{1}{2}\sum_{\mathclap{s\in Sh(2,1)}}e\left(s;f_1,f_2,f_3\right)
	e\left(f_{s_1}\right)e\left(f_{s_2}\right) 
		Y_{s_3}\wedge X_{D_2}\left(f_{s_1},f_{s_2}\right)\;,
\end{equation}
while for associated semi-Hamiltonian multivector fields 
$X_1$, $X_2$, $X_3$ an associated semi-Hamiltonian multivector field $X_{D_3}\left(f_1,f_2,f_3\right)$ 
is given by
\begin{equation}\label{ass_sHam_3}
\sum_{\mathclap{s\in Sh(2,1)}}e\left(s;f_1,f_2,f_3\right)
	\left(e\left(f_{s_2}\right)\left[\left[X_{s_3},X_{s_2}\right],Y_{s_1}\right]
	+ X_{s_3}\wedge X_{D_2}\left(f_{s_1},f_{s_2}\right)\right)\;.
\end{equation}
\end{theorem}
\begin{proof}To see that the definition does not depend on the particular chosen associated 
Hamiltonian multivector fields, we use (\ref{kernel_prop}) and proceed as in (\ref{well_def_2}). 

Graded symmetry follows from the graded symmetry of the semi-Hamiltonian multivector fields 
$X_{D_2}\left(f_i,f_j\right)$.

The contraction $i_{X}$ is graded linear of (symmetric) degree $|\,X\,|$ for any 
homogeneous multivector field $X$ and since $|\,X_{D_2}(f_i,f_j)\,|=deg(f_j)+deg(f_j)-1$ 
for any $i,j\in \N_3$, the sh-Lie $3$-bracket is homogeneous of degree $-1$.

To see the strong homotopy Jacobi equation in dimension three, we apply the definition
of the differential and of $D_3$ to rewrite the left side
\begin{multline*}
D_1D_3\left(f_1, f_2, f_3\right)=-dD_3\left(f_1,f_2,f_3\right)=\\
\textstyle \frac{1}{2}\sum_{s\in Sh(2,1)}e\left(s;f_1,f_2, f_3\right)
	e\left(f_{s_1}\right)e\left(f_{s_2}\right)
		di_{X_{D_2}\left(f_{s_1},f_{s_2}\right)}f_{s_3}\;.
\end{multline*}
Regarding (\ref{jacobi_id}) and (\ref{def_lie}) we use 
$e\left(D_2\left(f_i,f_j\right)\right)=-e(f_i)e(f_j)$ and insert appropriate 
correction terms
\begin{multline*}
\textstyle \frac{1}{2}\sum_{s\in Sh(2,1)}e\left(s;f_1,f_2, f_3\right)
	e\left(f_{s_1}\right)e\left(f_{s_2}\right)\\
\cdot\left(di_{X_{D_2}\left(f_{s_1},f_{s_2}\right)}f_{s_3}
	+e\left(f_{s_1}\right)e\left(f_{s_2}\right)
		i_{X_{D_2}\left(f_{s_1},f_{s_2}\right)}df_{s_3}\right)\\
			-\textstyle \frac{1}{2}\sum_{s\in Sh(2,1)}e\left(s;f_1,f_2, f_3\right)
				i_{X_{D_2}\left(f_{s_1},f_{s_2}\right)}df_{s_3}
\end{multline*}
to rewrite the first shuffle sum into a sum over Lie derivations 
\begin{multline*}
\textstyle \frac{1}{2}\sum_{s\in Sh(2,1)}e\left(s;f_1,f_2, f_3\right)
	e\left(f_{s_1}\right)e\left(f_{s_2}\right)
		L_{X_{D_2}\left(f_{s_1},f_{s_2}\right)}f_{s_3}\\
			-\textstyle \frac{1}{2}\sum_{s\in Sh(2,1)}e\left(s;f_1,f_2, f_3\right)
	i_{X_{D_2}\left(f_{s_1},f_{s_2}\right)}df_{s_3}\;.
\end{multline*}
According to (\ref{jacobi_id}) the first shuffle sum is just the negative Jacobi
expression and rewriting the latter we get 
\begin{multline*}
-\textstyle \sum_{s\in Sh(2,1)}e\left(s;f_1,f_2, f_3\right)
	D_2\left(D_2\left(f_{s_1},f_{s_2}\right),f_{s_3}\right)\\
		-\textstyle \frac{1}{2}\sum_{s\in Sh(2,1)}e\left(s;f_{s_1},f_{s_2}, df_{s_3}\right)
			e\left(f_{s_1}\right)e\left(f_{s_2}\right)
				i_{X_{D_2}\left(f_{s_1},f_{s_2}\right)}df_{s_3}
\end{multline*}
which we an rewrite into
\begin{multline*}
-\textstyle \sum_{s\in Sh(2,1)}e\left(s;f_1,f_2, f_3\right)
	D_2\left(D_2\left(f_{s_1},f_{s_2}\right),f_{s_3}\right)\\
-\textstyle \sum_{s\in Sh(2,1)}e\left(s;f_1,f_2, f_3\right)
	D_3\left(f_{s_1},f_{s_2},D_1f_{s_3}\right)\;.
\end{multline*}

To see that $Y_{D_3}\left(f_1,f_2,f_3\right)$ is a Hamiltonian multivector field
associated to the image $D_3\left(f_1,f_2,f_3\right)$, just 
apply the contraction of $\omega$ along $Y_{D_3}$ using (\ref{contract}). 

To see that $X_{D_3}\left(f_1,f_2,f_3\right)$ is a semi-Hamiltonian multivector field
associated to $D_3\left(f_1,f_2,f_3\right)$, we write 
$Y:= \sum_{s\in Sh\left(2,1\right)}e\left(s;f_1,f_2,f_3\right)
	e\left(f_{s_2}\right)[[X_{s_3},X_{s_2}],Y_{s_1}]$ and use (\ref{J_2}) to compute:
\begin{align*}
i_{X_{D_{3}\left(f_1,f_2,f_3\right)}}\omega
&= i_Y\omega +\textstyle\sum_{s\in Sh\left(2,1\right)}e\left(s;f_1,f_2,f_3\right)
			i_{X_{s_3}\wedge X_{D_2}\left(f_{s_1},f_{s_2}\right)}\omega\\
&=-\textstyle \sum_{s\in Sh\left(2,1\right)}e\left(s;f_1,f_2, f_3\right)
					D_2\left(D_2\left(f_{s_1},f_{s_2}\right),f_{s_3}\right)\\
&\quad+\textstyle\sum_{s\in Sh\left(2,1\right)}e\left(s;f_1,f_2,f_3\right)
			i_{X_{s_3}\wedge X_{D_2}\left(f_{s_1},f_{s_2}\right)}\omega\\
&=-\textstyle \sum_{s\in Sh\left(2,1\right)}e\left(s;f_1,f_2, f_3\right)
					D_2\left(D_2\left(f_{s_1},f_{s_2}\right),f_{s_3}\right)\\
&\quad+\textstyle\sum_{s\in Sh\left(1,2\right)}e\left(s;f_1,f_2,f_3\right)
			i_{X_{D_2}\left(f_{s_1},f_{s_2}\right)}df_{s_3}\\
&=-\textstyle\sum_{s\in Sh\left(2,1\right)}e\left(s;f_1,f_2, f_3\right)
					D_2\left(D_2\left(f_{s_1},f_{s_2}\right),f_{s_3}\right)\\
&\quad+\textstyle\sum_{s\in Sh\left(1,2\right)}e\left(s;f_1,f_2,f_3\right)
			D_3\left(df_{s_1},f_{s_2},f_{s_3}\right)\\
&=D_1D_3\left(f_1,f_2,f_3\right) \; .			
\end{align*}
\end{proof}
\begin{remark}At this point we should stress again, that if there is 
\textit{no} Hamiltonian multivector field $Y$ satisfying 
$$i_Y\omega = -\textstyle\sum_{s\in Sh\left(2,1\right)}e\left(s;f_1,f_2,f_3\right)
	D_2\left(D_2\left(f_{s_1},f_{s_2}\right),f_{s_3}\right)\; ,
$$
then the previous proof shows, that the image $D_3(f_1,f_2,f_3)$ is not semi-Hamiltonian. 
Regarding example (\ref{counterexample_1}) this justifies our definition of Hamiltonian 
forms as differential forms satisfying \textbf{both} fundamental equations.
\end{remark}

\subsection{The general multilinear Operator}We give an inductive definition of 
$k$-linear operators, based at the sh-Lie $3$-bracket from the previous
section. These operators satisfy strong homotopy Jacobi equations in any dimension
and define a Lie $\I$-algebra structure on the set of Hamiltonian forms.
\begin{definition}\label{niceBracket}
Let $\left(M, \omega\right)$ be an n-plectic manifold and $H\left(M\right)$ the $\Z$-graded 
vector space of Hamiltonian forms on $M$. The \textbf{strong homotopy Lie k-bracket}
\begin{equation}
D_{k}: 	H(M)\times \cdots \times H(M) \to  H(M) \; ,
\end{equation}
is defined inductively for any $k>3$, homogeneous $f_1,\ldots,f_k\in H\left(M\right)$ 
and semi-Hamiltonian multivector fields $X_{D_{k-1}}\left(\cdot,\cdots,\cdot\right)$ 
associated to the strong homotopy Lie $(k-1)$-bracket $D_{k-1}$ by  
$$
D_k(f_1,\ldots ,f_{k}):=
-\sum_{\mathclap{s\in Sh(k-1,1)}}e\left(s;f_1,\ldots,f_{k}\right)
	e\left(f_{s_1}\right)\cdots e\left(f_{s_{k-1}}\right) 
		i_{X_{D_{k-1}}\left(f_{s_1},\ldots,f_{s_{k-1}}\right)}f_{s_{k}}
$$ 
and is then extended to $H\left(M\right)$ by linearity. 
\end{definition}
The induction base is the sh-Lie $3$-bracket. If we referee to the sh-Lie $k$-bracket for any $k\in\N$, then the differential $D_1$ is meant to be the sh-Lie $1$-bracket.

The following theorem is the central statement in this work and basically says, that the
sequence of sh-Lie $k$-brackets defines a Lie $\I$-algebra on the vector space of Hamiltonian
forms.
\begin{theorem}\label{main_theorem}
The operator $D_k$ is well defined, graded symmetric and homogeneous of degree $-1$ for any
$k\in\N$ and the strong homotopy Jacobi equation 
$$
\sum\nolimits_{i+j=n+1} \left(\sum\nolimits_{s\in Sh(j,i-1)}
e(s;f_1,\ldots,f_n)D_i\left(D_j \left(
f_{s_1}, \ldots, f_{s_j} \right), f_{s_{j+1}}, \ldots, f_{s_n}\right)\right) = 0
$$
is satisfied for any Hamiltonian forms $f_1,\ldots,f_n \in H(M)$ and in any dimension 
$n\in\N$. 

If $Y_1,\ldots,Y_k$ are associated Hamiltonian multivector fields,
an associated Hamiltonian multivector field 
$Y_{D_k}\left(f_1,\ldots,f_k\right)$ is given by
\begin{equation}\label{ass_Ham_k}
\sum_{\mathclap{s\in Sh(k-1,1)}}e\left(s;f_1,\ldots,f_{k}\right)
	e\left(f_{s_1}\right)\cdots e\left(f_{s_{k-1}}\right) 
		Y_{s_k}\wedge X_{D_{k-1}}\left(f_{s_1},\ldots,f_{s_{k-1}}\right)\;.
\end{equation}
If $X_1,\ldots,X_k$ are associated semi-Hamiltonian multivector field,
an associated semi-Hamiltonian multivector field
$X_{D_k}\left(f_1,\ldots,f_k\right)$ 
is given by
\begin{equation}\label{ass_sHam_k}
\sum_{\mathclap{s\in Sh(k-1,1)}}e\left(s;f_1,\ldots,f_{k}\right)
X_{s_{k}}\wedge X_{D_{k-1}}\left(f_{s_1},\ldots,f_{s_{k-1}}\right)
-Y_{J_k}\left(f_{s_1},\ldots,f_{s_{k}}\right)\;,
\end{equation} 
where the $Y_{J_k}(f_{s_1},\ldots,f_{s_{k+1}})$ is defined by the
equation
$$
i_{Y_{J_k}(f_{s_1},\ldots,f_{s_{k+1}})}\omega = 
-\sum^{i,j>1}_{i+j=k+1} (\sum_{s\in Sh(j,k-j)}
e(s)D_i\left(D_j \left(
f_{s_1}, \ldots, f_{s_j} \right), f_{s_{j+1}}, \ldots, f_{s_k}\right))
$$
\end{theorem}
\begin{proof}(By induction)
For $k\leq 3$ this was shown in the previous sections. For the induction step assume
that all statements of the theorem are true for some $k\in\N$. We proof that they are
true for $(k+1)$:

First of all lets see that the definition does not depend on the particular chosen
associated semi-Hamiltonian multivector fields. This follows from proposition
(\ref{kernel_prop}). Since the difference of multivector fields associated to the 
same Hamilton form differ only in elements of the kernel of $\omega$ we can write
$
i_{X'_{D_{k}}\left(f_{1},\ldots,f_{k}\right)}f_{k+1}=
i_{X_{D_{k}}\left(f_{1},\ldots,f_{k}\right)+\xi}f_{k+1}=
i_{X_{D_{k}}\left(f_{1},\ldots,f_{k}\right)}f_{k+1}
$ because each $f_i$ has the kernel property. 

To see that $D_{k+1}$ is graded symmetric we use the assumed symmetry of any associated 
semi-Hamiltonian multivector field $X_{D_k}\left(f_1,\ldots,f_k\right)$ (up to elements of the kernel of $\omega$) and rewrite the expression
$D_{k+1}(f_1,\ldots,f_{k+1})$ in terms of the symmetric group like
$$
\textstyle\frac{1}{k!}\sum_{s\in S_{k+1}}e\left(s;f_1,\ldots,f_k\right)
e\left(D_{k}\left(f_{s_1},\ldots,f_{s_{k}}\right)\right) 
			i_{X_{D_{k}}\left(f_{s_1},\ldots,f_{s_{k}}\right)}f_{s_{k+1}}\; ,
$$
which is graded symmetric.

To see that the operator is homogeneous of degree $-1$, assume that every 
argument is homogeneous. Then the degree
$deg(i_{X_{D_{k}(f_{1},\ldots,f_{k})}}f_{k+1})=\sum deg(f_j)-1$ follows from the
assumption that $D_{k}\left(f_1,\ldots,f_k\right)$ is homogeneous of degree $-1$. 

The proof of the strong homotopy Jacobi equation is a very long 
calculations. According to a better readable text, we put it into appendix
(\ref{main_proof}).

To see that (\ref{ass_Ham_k}) is a Hamiltonian multivector field associated 
to $D_k(f_1,\ldots,f_k)$ just compute $i_{Y_{D_k}\left(f_1,\ldots,f_k\right)}\omega$.

To see that (\ref{ass_sHam_k}) is a semi-Hamiltonian multivector field associated 
to $D_k(f_1,\ldots,f_k)$ we compute
\begin{align*}
&i_{X_{D_k}\left(f_1,\ldots,f_k\right)}\omega\\ 
&=\textstyle\sum_{s\in Sh(k-1,1)}e\left(s;f_1,\ldots,f_{k}\right)
	i_{X_{s_{k}}\wedge X_{D_{k-1}}\left(f_{s_1},\ldots,f_{s_{k-1}}\right)}\omega
		-i_{Y_{J_k}\left(f_{s_1},\ldots,f_{s_{k}}\right)}\omega\\
&=\textstyle\sum_{s\in Sh(k-1,1)}e\left(s;f_1,\ldots,f_{k}\right)
	i_{X_{D_{k-1}}\left(f_{s_1},\ldots,f_{s_{k-1}}\right)}df_{s_k}\\
&\quad-\textstyle\sum^{i,j>1}_{i+j=k+1}\sum_{s\in Sh(j,k-j)}e\left(s;f_1,\ldots,f_{k}\right)
	D_i\left(D_j \left(f_{s_1}, \ldots, f_{s_j} \right), f_{s_{j+1}}, \ldots, f_{s_k}\right)\\
&=-\textstyle\sum^{i>1}_{i+j=k+1}\sum_{s\in Sh(j,k-j)}e\left(s;f_1,\ldots,f_{k}\right)
	D_i\left(D_j \left(f_{s_1}, \ldots, f_{s_j} \right), f_{s_{j+1}}, \ldots, f_{s_k}\right)
\end{align*}
And since the strong homotopy Jacobi equation is satisfied in dimension $k$ the last sum equals
$D_1D_k(f_1,\ldots,f_k)$ and $X_{D_k}\left(f_1,\ldots,f_k\right)$ is a solution to the first
fundamental equation.
\end{proof}
Since the set of Hamiltonian forms is an $\Z$-graded vector space an immediate
consequence is 
\begin{corollary}[The Lie $\I$-Algebra of Hamiltonian Forms]
Let $\left(M,\omega\right)$ be an $n$-plectic manifold and $H(M)$ the
$\Z$-graded vector space of Hamiltonian forms. The sequence 
$(D_k)_{k\in \N}$ of sh-Lie $k$-brackets defines a Lie $\I$-algebra on
$H(M)$.
\end{corollary}
\section{Conclusion and Outlook}
We defined an Lie $\I$-algebra on Hamiltonian forms on any 
$n$-plectic manifold $M$. However since differential forms are moreover
sections of a vector bundle, the question arises, whether or not
Hamiltonian forms are vector bundle sections and hence have a Lie $\I$-algebroid 
structure in addition.

As seen in (\ref{negative_module}) the answer is not trivial, since Hamiltonian forms
are at least not a $C^\I(M)$-module, but a $C_\omega^\I(M)$-module instead.

\begin{appendix}
\section{Proof of the sh-Jacobi equation in \ref{main_theorem}}\label{main_proof}
We proof the strong homotopy Jacobi equation in dimension $k$ under
the assumption that it is satisfied in dimension $(k-1)$. 
According to a better readable text we start with some auxiliary calculations, necessary
to keep the main part as simple as possible.

The computation is different for $k=4$ and will be treated after the general situation.
We assume $k>4$, that $f_1,\ldots,f_{k} \in H(M)$ 
are homogeneous Hamiltonian forms with associated
semi-Hamiltonian multivector fields $X_1,\ldots,X_{k}$ and that graded symmetric 
associated semi-Hamiltonian multivector fields of $D_{j}$ are given by (\ref{ass_sHam_k}) 
for any $j\smin \ordinal{k}$.

The first step is to rewrite the strong homotopy Jacobi equation, according to the definition 
of the sh-Lie $k$-bracket. Since the sh-Lie $3$-bracket differs from the higher
brackets by the factor $\frac{1}{2}$ only, we define $c_k:=1$ for any $k\neq 3$ and 
$c_3:=\frac{1}{2}$, to handle them equally. Nevertheless the sh-Lie $2$-bracket is 
different and we consider the situation $i>2$ first:
\begin{my_equation}\label{my_equation_1}
\begin{align*}
&\textstyle\sum_{s\in Sh(j,i-1)}e\left(s;f_1,\ldots,f_k\right)  
	D_{i}\left(D_{j}\left(f_{s_{1}},\ldots,f_{s_{j}}\right),
		f_{s_{j+1}},\ldots,f_{s_{k}}\right) =\\
&c_i\textstyle\sum_{s\in Sh(j,i-2,1)}e\left(s;f_1,\ldots,f_k\right)	
		e\left(f_{s_1}\right)\cdots e\left(f_{s_{k-1}}\right)
			i_{X_{D_{i-1}}\left(D_j\left(f_{s_1},\ldots,f_{s_j}\right),
				f_{s_{j+1}},\ldots,f_{s_{k-1}}\right)}f_{s_k}\\
&-c_i\textstyle\sum_{s\in Sh(i-1,j)}e\left(s;f_1,\ldots,f_k\right)
	i_{X_{D_{i-1}}\left(f_{1},\ldots,f_{i-1}\right)}D_j\left(f_{s_i},\ldots,f_{s_k}\right)
\end{align*}
for any $i,j\in \N$ with $i>2$ and $i+j=k+1$.
\end{my_equation}
\begin{proof} We can split the definition of the sh-Lie $k$-bracket into two parts. 
A summation over shuffles that fix the first element and a remaining term, where the first element is shuffled to the last position. 

In particular let
$Sh(\cdot,i,j)$ be the set of permutations 
$\left(1,\mu_1,\ldots,\mu_i,\nu_1\ldots,\nu_j\right)$,
subject to the conditions $\mu_1<\ldots<\mu_i$ and $\nu_1<\ldots<\nu_j$. Then
\begin{align*}
& D_i\left(f_1,\ldots, f_i\right)\\ 
&=-c_i\textstyle\sum_{s\in Sh(\cdot,i-2,1)}e(s;f_1,\ldots,f_i)
			e\left(f_1\right)e\left(f_{s_2}\right)\cdots e\left(f_{s_{i-1}}\right)
				i_{X_{D_{i-1}}\left(f_1,f_{s_2},\ldots,f_{s_{i-1}}\right)}f_{s_i}\\
&\quad-c_ie\left(\left(i,1,\ldots,i-1\right);f_1,\ldots,f_i\right)
			e\left(f_{2}\right)\cdots e\left(f_{i}\right)
				i_{X_{D_{i-1}}\left(f_{2},\ldots,f_{i}\right)}f_1\;.
\end{align*}
If we substitute the operator $D_j$ for the first argument of $D_i$,  
we can rewrite the left side of equation (\ref{my_equation_1}) into
\begin{align*}
&-c_i\textstyle\sum_{s\in Sh(j,i-1)}e\left(s;f_1,\ldots,f_k\right)\left[
	\sum_{t\in Sh(\cdot,i-2,1)}
		e\left(t;D_j\left(f_{s_1},\ldots,f_{s_j}\right),f_{s_{j+1}}
			\ldots,f_{s_k}\right)\right.\\
&\quad\cdot	
			e\left(D_j\left(f_{s_1},\ldots,f_{s_j}\right)\right)
				e\left(f_{t\circ s_{j+1}}\right)\cdots e\left(f_{t\circ s_{k-1}}\right)\\
&\quad\cdot	\left.i_{X_{D_{i-1}}\left(D_j\left(f_{s_1},\ldots,f_{s_j}\right),
							f_{t\circ s_{j+1}},\ldots,f_{t\circ s_{k-1}}\right)}f_{t\circ s_k}\right]\\
&\quad-c_i\textstyle\sum_{s\in Sh(j,i-1)}e\left(s;f_1,\ldots,f_k\right)
	e\left(\left(i,1,2\ldots,i-1\right);D_j\left(f_{s_1},\ldots,f_{s_j}\right),
		f_{s_{j+1}},\ldots,f_{s_k}\right)\\
&\quad\cdot e\left(f_{s_{j+1}}\right)\cdots e\left(f_{s_{k}}\right)
				i_{X_{D_{i-1}}\left(f_{s_{j+1}},\ldots,f_{s_k}\right)}D_j\left(f_{s_1},\ldots,f_{s_j}\right).
\end{align*}
To simplify this, apply 
$e\left(D_j\left(f_{s_1},\ldots,f_{s_j}\right)\right)=-e(f_{s_1})\cdots e(f_{s_j})$
and the bijective map
\begin{multline*}
Sh(j,i-1)\times Sh(\cdot,i-2,1) \to Sh(j,i-2,1)\\
\left((\mu_1,\ldots,\mu_j,\nu_1,\ldots,\nu_{i-1}),
(1,\lambda_1,\ldots,\lambda_{i-2},\kappa)\right)\mapsto\\
(id_{S_j}\times(\lambda_1,\ldots,\lambda_{i-2},\kappa))\circ
(\mu_1,\ldots,\mu_j,\nu_1,\ldots,\nu_{i-1}),
\end{multline*}
on the first part, expand
$e\left(\left(2,\ldots,i,1\right);D_j\left(f_{s_1},\ldots,f_{s_j}\right),
f_{s_{j+1}},\ldots,f_{s_k}\right)$ according to
\begin{multline*}
e\left(\left(2,\ldots,i,1\right);D_j\left(f_{s_1},\ldots,f_{s_j}\right),
f_{s_{j+1}},\ldots,f_{s_k}\right)=\\
e\left(D_j\left(f_{s_1},\ldots,f_{s_j}\right),f_{s_{j+1}}\right)\cdots
e\left(D_j\left(f_{s_1},\ldots,f_{s_j}\right),f_{s_{k}}\right)=\\
e(f_{s_1},f_{s_{j+1}})\cdots e(f_{s_j},f_{s_{j+1}})e(f_{s_{j+1}})\cdots
e(f_{s_1},f_{s_{k}})\cdots e(f_{s_j},f_{s_{k}})e(f_{s_{k}})
\end{multline*}
and use the bijective map
$$Sh(j,i-1)\to Sh(i-1,j);
(\mu_1,\ldots,\mu_j,\nu_1,\ldots,\nu_{i-1})\mapsto 
(\nu_1,\ldots,\nu_{i-1},\mu_1,\ldots,\mu_{j})
$$
on the second part to rewrite the substituted expression into the right side of equation
(\ref{my_equation_1}).
\end{proof}
The case $i\leq 2$ will be treated later. First we applying the definition of $D_j$ to expend equation (\ref{my_equation_1}) further. Again, since $D_2$ is different, we additionally assume $j>2$ and get
\begin{my_equation}\label{my_equation_2}
\begin{align*}
&\textstyle\sum_{s\in Sh(j,i-1)}e\left(s;f_1,\ldots,f_{k}\right)
	D_{i}\left(D_{j}\left(f_{s_{1}},\ldots,f_{s_{j}}\right),
		f_{s_{j+1}},\ldots,f_{s_{k}}\right)\\
&=c_i\textstyle\sum_{s\in Sh(j,i-2,1)}e\left(s;f_1,\ldots,f_{k}\right)
	e\left(f_{s_{1}}\right)\cdots e\left(f_{s_{k-1}}\right)
		i_{X_{D_{i-1}}\left(D_{j}\left(f_{s_{1}},\ldots,f_{s_{j}}\right)
			\ldots,f_{s_{k-1}}\right)}f_{s_{k}}\\
&\quad+c_ic_j\textstyle\sum_{s\in Sh(i-1,j-1,1)}e\left(s;f_1,\ldots,f_k\right)
	e(f_{s_i})\cdots e(f_{s_{k-1}})\\
&\quad\quad\cdot i_{X_{D_{j-1}}\left(f_{s_i},\ldots,f_{s_{k-1}}\right)\wedge
			X_{D_{i-1}}\left(f_{s_1},\ldots,f_{s_{i-1}}\right)}f_{s_k}
\end{align*}
for any $i\,,j>2$ and $i+j=k+1$.
\end{my_equation}
This is obvious from equation (\ref{my_equation_1}) and we skip the calculation.
We can exploit the symmetry of the wedge product in the previous expression
when we sum over over all possible combinations of $i$ and $j$. To be more precise
\begin{my_equation}\label{my_equation_3}
\begin{multline*}
\textstyle\sum_{i+j=k+1}^{i,j>2}c_ic_j\sum_{s\in Sh(i-1,j-1,1)}
	e\left(s;f_1,\ldots,f_k\right)\\
		\cdot e\left(f_{s_i}\right)\cdots e\left(f_{s_{k-1}}\right)
			i_{X_{D_{j-1}}\left(f_{s_i},\ldots,f_{s_{k-1}}\right)\wedge
				X_{D_{i-1}}\left(f_{s_1},\ldots,f_{s_{i-1}}\right)}f_{s_k} = 0
\end{multline*}
\end{my_equation}
\begin{proof}For any shuffle 
$s:=(\mu_1,\ldots,\mu_{i-1},\nu_1,\ldots,\nu_{j-1},\delta)\in Sh(i-1,j-1,1)$ there
is exactly one shuffle 
$s^*:=(\nu_1,\ldots,\nu_{j-1},\mu_1,\ldots,\mu_{i-1},\delta)\in Sh(j-1,i-1,1)$ and
the identity $s^{**}=s$ holds. From the graded symmetry of the wedge product 
we get
\begin{align*}
&e\left(s;f_1,\ldots,f_k\right)
	e\left(f_{\nu_{1}}\right)\cdots e\left(f_{\nu_{j-1}}\right)
		i_{X_{D_{j-1}}\left(f_{\nu_{1}},\ldots,f_{\nu_{j-1}}\right)\wedge 
			X_{D_{i-1}}\left(f_{\mu_1},\ldots,f_{\mu_{i-1}}\right)}f_{\delta}\\
&= e\left(s;f_1,\ldots,f_k\right)
	e\left(f_{\nu_{1}}\right)\cdots e\left(f_{\nu_{j-1}}\right)
		e\left(D_{i-1}\left(f_{\mu_{1}},\ldots,f_{\mu_{i-1}}\right),
			D_{j-1}\left(f_{\nu_{i}},\ldots,f_{\nu_{j-1}}\right)\right)\\
&\quad\cdot	i_{X_{D_{i-1}}\left(f_{\mu_{1}},\ldots,f_{\mu_{i-1}}\right)\wedge 
			X_{D_{j-1}}\left(f_{s_{\nu_1}},\ldots,f_{\nu_{j-1}}\right)}f_{\delta}\\
&=-e\left(s^*;f_1,\ldots,f_k\right)
	e\left(f_{\mu_{1}}\right)\cdots e\left(f_{\mu_{i-1}}\right)
		i_{X_{D_{i-1}}\left(f_{\mu_{1}},\ldots,f_{\mu_{i-1}}\right)\wedge 
			X_{D_{j-1}}\left(f_{\nu_1},\ldots,f_{\nu_{j-1}}\right)}f_{\delta}
\end{align*}
and consequently each terms indexed by $s$ cancel with the unique term indexed
by $s^*$ in (\ref{my_equation_3}).
\end{proof}
Now we look at the situation where either $i=2$ or $j=2$. As it turns out
it is advantageous to join them in a single equation. Since $k>4$ and $i+j=k+1$ we get
\begin{my_equation}\label{my_equation_4}
\begin{align*}
&\textstyle\sum_{s\in Sh(k-1,1)}e\left(s;f_1,\ldots,f_k\right)
	D_{2}\left(D_{k-1}\left(f_{s_{1}},\ldots,f_{s_{k-1}}\right),f_{s_{k}}\right)\\
&\quad +\textstyle\sum_{s\in Sh(2,k-2)}e\left(s;f_1,\ldots,f_k\right)
	D_{k-1}\left(D_{2}\left(f_{s_{1}},f_{s_{2}}\right),f_{s_{3}},\ldots,f_{s_{k}}\right)\\
&=-\textstyle\sum_{s\in Sh(k-1,1)}e\left(s;f_1,\ldots,f_k\right)
	e(f_{s_{1}})\cdots e(f_{s_{k-1}})
		L_{X_{D_{k-1}}\left(f_{s_{1}},\ldots,f_{s_{k-1}}\right)}f_{s_{k}}\\
&\quad+\textstyle\frac{1}{2}\sum_{s\in Sh(k-2,1,1)}e\left(s;f_1,\ldots,f_k\right)
	e\left(f_{s_{1}}\right)\cdots e\left(f_{s_{k-1}}\right)
		i_{X_{D_{2}}\left(D_{k-2}\left(f_{s_{1}},\cdots,f_{s_{k-2}}\right),
			f_{s_{k-1}}\right)}f_{s_{k}}\\
&\quad+\textstyle\sum_{s\in Sh(2,k-3,1)}e\left(s;f_1,\ldots,f_k\right)
	e\left(f_{s_{1}}\right)\cdots e\left(f_{s_{k-1}}\right)
		i_{X_{D_{k-2}}\left(D_{2}\left(f_{s_{1}},f_{s_{2}}\right),f_{s_{3}},\ldots,
			f_{s_{k-1}}\right)}f_{s_{k}}
\end{align*}
\end{my_equation}
\begin{proof} We transform the first shuffle sum according to the
definition of the sh-Lie $2$-bracket into
\begin{multline*}
\textstyle\sum_{s\in Sh(k-1,1)}e\left(s;f_1,\ldots,f_k\right)
	D_{2}\left(D_{k-1}\left(f_{s_{1}},\cdots,f_{s_{k-1}}\right),f_{s_{k}}\right)=\\
-\textstyle\sum_{s\in Sh(k-1,1)}e\left(s;f_1,\ldots,f_k\right)
	e\left(f_{s_1}\right)\cdots e\left(f_{s_{k-1}}\right)
		L_{X_{D_{k-1}}\left(f_{s_{1}},\cdots,f_{s_{k-1}}\right)}f_{s_{k}}\\	
+\textstyle\sum_{s\in Sh(k-1,1)}e\left(s;f_1,\ldots,f_k\right)
		e\left(D_{k-1}\left(f_{s_{1}},\ldots,f_{s_{k-1}}\right),f_{s_{k}}\right)
			e\left(f_{s_{k}}\right)\\
				\cdot L_{X_{s_{k}}}
					D_{k-1}\left(f_{s_{1}},\cdots,f_{s_{k-1}}\right)\;,
\end{multline*}
use the map 
$Sh(k-1,1)\to Sh(1,k-1);(\mu_1,\ldots,\mu_{k-1},\nu)\mapsto(\nu,\mu_1,\ldots,\mu_{k-1})$
and the identity 
$e\left(D_{k-1}\left(f_{s_{1}},\ldots,f_{s_{k-1}}\right),f_{s_{k}}\right)
e\left(f_{s_{k}}\right)= e(f_{s_1},f_{s_{k}})\cdots e(f_{s_{k-1}},f_{s_{k}})$
to canonicalize the expression
\begin{multline*}
-\textstyle\sum_{s\in Sh(k-1,1)}e\left(s;f_1,\ldots,f_k\right)
	e\left(f_{s_{1}}\right)\cdots e\left(f_{s_{k-1}}\right)
		L_{X_{D_{k-1}}\left(f_{s_{1}},\ldots,f_{s_{k-1}}\right)}f_{s_{k}}\\
+\textstyle\sum_{s\in Sh(1,k-1)}e\left(s;f_1,\ldots,f_k\right)
	L_{X_{s_{1}}}D_{k-1}\left(f_{s_{2}},\ldots,f_{s_{k}}\right)\;.
\end{multline*}
Then apply the definition of $D_{k-1}$, 
taking $(k-1)\cdot|Sh(1,k-1)|=|Sh(1,k-2,1)|$ into account and 
rewrite the second shuffle sum using equation (\ref{my_equation_1}) to combine
both
\begin{align*}
&-\textstyle\sum_{s\in Sh(k-1,1)}e\left(s;f_1,\ldots,f_k\right)
	e\left(f_{s_{1}}\right)\cdots e\left(f_{s_{k-1}}\right)
		L_{X_{D_{k-1}}\left(f_{s_{1}},\cdots,f_{s_{k-1}}\right)}f_{s_{k}}\\
&-\textstyle\sum_{s\in Sh(1,k-2,1)}e\left(s;f_1,\ldots,f_k\right)
	e\left(f_{s_{2}}\right)\cdots e\left(f_{s_{k-1}}\right)L_{X_{s_{1}}}
		i_{X_{D_{k-2}}\left(f_{s_{2}},\cdots,f_{s_{k-1}}\right)}f_{s_{k}}\\
&+\textstyle\sum_{s\in Sh(2,k-3,1)}e\left(s;f_1,\ldots,f_k\right)	
			e\left(f_{s_1}\right)\cdots e\left(f_{s_{k-1}}\right)
			i_{X_{D_{k-2}}\left(D_2\left(f_{s_1},f_{s_2}\right),
							f_{s_{3}},\ldots,f_{s_{k-1}}\right)}f_{s_k}\\
&-\textstyle\sum_{s\in Sh(k-2,1,1)}e\left(s;f_1,\ldots,f_k\right)e(f_{s_{k-1}})
	i_{X_{D_{k-2}}\left(f_{s_1},\ldots,f_{s_{k-2}}\right)}L_{X_{s_{k-1}}}f_{s_{k}}.
\end{align*}
Using $Sh(1,k-2,1)\to Sh(k-2,1,1)\,;\, (\lambda,\mu_1,\ldots,\mu_{k-2},\nu)
\mapsto (\mu_1,\ldots,\mu_{k-2},\lambda,\nu)$ we reorder the first shuffle sum
and join it with the last shuffle sum, applying the second equation in (\ref{multi_rules})
to get:
\begin{align*}
&\textstyle\sum_{s\in Sh(k-2,1,1)}e\left(s;f_1,\ldots,f_k\right)e\left(f_{s_{k-1}}\right)
	i_{[X_{s_{k-1}},X_{D_{k-2}}\left(f_{s_{1}},\ldots,f_{s_{k-2}}\right)]}f_{s_{k}}\\
&-\textstyle\sum_{s\in Sh(k-1,1)}e\left(s;f_1,\ldots,f_k\right)
	e\left(f_{s_{1}}\right)\cdots e\left(f_{s_{k-1}}\right)
		L_{X_{D_{k-1}}\left(f_{s_{1}},\ldots,f_{s_{k-1}}\right)}f_{s_{k}}\\
&+\textstyle\sum_{s\in Sh(2,k-3,1)}e\left(s;f_1,\ldots,f_k\right)	
			e\left(f_{s_1}\right)\cdots e\left(f_{s_{k-1}}\right)
			i_{X_{D_{k-2}}\left(D_2\left(f_{s_1},f_{s_2}\right),
							f_{s_{3}},\ldots,f_{s_{k-1}}\right)}f_{s_k}.
\end{align*}
Finall, the right side of equation (\ref{my_equation_2}) follows from the definition 
of the associated semi-Hamiltonian multivector field
$X_{D_2}(X_n,X_m)$, for appropriate $n,m\in \N$.
\end{proof}
Terms where either $i=1$ or $j=1$ needs no precalculation. Nevertheless we need the
sh-Jacobi equation in its semi-Hamiltonian multivector field incarnation:
\begin{prop}\label{_mvf_sh_J}Suppose that the strong homotopy Jacobi equation is
satisfied in dimension $(k-1)$. Then there is a multivector field $\xi \in \ker(\omega)$ with 
\begin{flalign*}
&\textstyle\sum_{i+j=k}^{j>1,i>1}
		\sum_{s\in Sh\left(j,i-1\right)}e\left(s;f_1,\ldots,f_k\right)
			X_{D_{i}}\left(D_{j}\left(f_{s_{1}},\ldots,f_{s_{j}}\right),
				f_{s_{j+1}},\ldots,f_{s_{k-1}}\right) & \\
&=-\textstyle\sum_{s\in Sh\left(1,k-2\right)}e\left(s;f_1,\ldots,f_k\right) 
	X_{D_{k-1}}\left(D_1f_{s_1},\ldots ,f_{k-1}\right)+\xi &		
\end{flalign*}
\end{prop}
\begin{proof}
Apply the n-plectic differential $D_1$ to the strong homotopy 
Jacobi equation in dimension $\left(k-1\right)$. Since $D_1D_1=0$ we get 
\begin{align*}
&\textstyle\sum_{i+j=k}^{j>1,i>1}
	\sum_{s\in Sh(j,i-1)}e\left(s;f_1,\ldots,f_k\right)
		D_1D_{i}\left(D_{j}\left(f_{s_{1}},\cdots,f_{s_{j}}\right),
			f_{s_{j+1}},\cdots,f_{s_{k-1}}\right)\\
&=- \textstyle\sum_{s\in Sh(1,k-2)}e\left(s;f_1,\ldots,f_k\right)
	D_1D_{k-1}\left(D_1f_{s_{1}},\ldots ,f_{s_{k-1}}\right)
\end{align*}
and using the fundamental pairing this transforms into
\begin{align*}
&\textstyle\sum_{i+j=k}^{j>1,i>1}
	\sum_{s\in Sh(j,i-1)}e\left(s;f_1,\ldots,f_k\right)
		i_{X_{D_{i}}\left(D_{j}\left(f_{s_{1}},\cdots,f_{s_{j}}\right),
			f_{s_{j+1}},\cdots,f_{s_{k-1}}\right)}\omega \\
&=- \textstyle\sum_{s\in Sh(1,k-2)}e\left(s;f_1,\ldots,f_k\right)
	i_{X_{D_{k-1}}\left(D_1f_{s_{1}},\ldots ,f_{s_{k-1}}\right)}\omega
\end{align*}
By proposition (\ref{kernel_prop}), the multivector fields on both sides of the equation 
only differ by an element of the kernel of $\omega$.
\end{proof}
Now we have all we need to calculate the strong homotopy Jacobi equation in dimension
$k$. Again, since the definition of the differential $D_1$ and the sh-Lie 2-bracket 
is different from the general situation, we separate appropriate terms according to:
\begin{align*}
&\textstyle{\sum_{i+j= k+1}}\left( \sum_{s\in Sh(j,i-1)}e\left(s;f_1,\ldots,f_k\right)	
		D_{i}\left(D_{j}\left(f_{s_{1}},\ldots,f_{s_{j}}\right),f_{s_{j+1}},\ldots,f_{s_{k}}
			\right)\right)&\\
&= D_{1}D_{k}\left(f_{1},\ldots,f_{k}\right)\\
&\quad+ \textstyle\sum_{s\in Sh(1,k-1)}e\left(s;f_1,\ldots,f_k\right)
	D_{k}\left(D_{1}f_{s_{1}},f_{s_{2}},\ldots,f_{s_{k}}\right)\\
&\quad+ \textstyle\sum_{s\in Sh(k-1,1)}e\left(s;f_1,\ldots,f_k\right)
	D_{2}\left(D_{k-1}\left(f_{s_{1}},\ldots,f_{s_{k-1}}\right),f_{s_{k}}\right)\\
&\quad+\textstyle\sum_{s\in Sh(2,k-2)}e\left(s;f_1,\ldots,f_k\right)
	D_{k-1}\left(D_{2}\left(f_{s_{1}},f_{s_{2}}\right),f_{s_{3}},
		\ldots,f_{s_{k}}\right)\\
&\quad+ \textstyle\sum_{i+j=k+1}^{i,j>2}\left( \sum_{s\in Sh(j,i-1)}
		e\left(s;f_1\ldots,f_k\right)	
			D_{i}\left(D_{j}(f_{s_{1}},...,f_{s_{j}}),f_{s_{j+1}},...,f_{s_{k}}
				\right)\right)\;.
\intertext{Applying the definition of $D_1$ and $D_k$ and using the previously calculated
expressions accordingly we rewrite this into}
&-\textstyle\sum_{s\in S(k-1,1)}e\left(s;f_1,\ldots,f_k\right)
	e\left(f_{s_{1}}\right)\cdots e\left(f_{s_{k-1}}\right)
		D_1i_{X_{D_{k-1}}\left(f_{s_{1}},\ldots,f_{s_{k-1}}\right)}f_{s_{k}}\\
&\quad+\textstyle\sum_{s\in Sh(1,k-2,1)}e\left(s;f_1,\ldots,f_k\right)
	e\left(f_{s_1}\right)\cdots e\left(f_{s_{k-1}}\right)
		i_{X_{D_{k-1}}\left(D_1f_{s_{1}},\ldots,f_{s_{k-1}}\right)}f_{s_{k}}\\
&\quad- \textstyle\sum_{s\in Sh(k-1,1)}e\left(s;f_1,\ldots,f_k\right)
			i_{X_{D_{k-1}}\left(f_{s_{1}},\ldots,f_{s_{k-1}}\right)}D_1f_{s_{k}}\\	
&\quad-\textstyle\sum_{s\in Sh(k-1,1)}e\left(s;f_1,\ldots,f_k\right)
	e\left(f_{s_{1}}\right)\cdots e\left(f_{s_{k-1}}\right)
		L_{X_{D_{k-1}}\left(f_{s_{1}},\cdots,f_{s_{k-1}}\right)}f_{s_{k}}\\
&\quad+ \textstyle\frac{1}{2}\sum_{s\in Sh(k-2,1,1)}e\left(s;f_1,\ldots,f_k\right)
	e\left(f_{s_{1}}\right)\cdots e\left(f_{s_{k-1}}\right)
		i_{X_{D_{2}}\left(D_{k-2}\left(f_{s_{1}},\dots,f_{s_{k-2}}\right),
		 f_{s_{k-1}}\right)}f_{s_{k}}\\
&\quad+\textstyle\sum_{s\in Sh(2,k-3,1)}e\left(s;f_1,\ldots,f_k\right)
	e\left(f_{s_{1}}\right)\cdots e\left(f_{s_{k-1}}\right)
		i_{X_{D_{k-2}}\left(D_{2}\left(f_{s_{1}},f_{s_{2}}\right),
			f_{s_{3}},\cdots,f_{s_{k-1}}\right)}f_{s_{k}}\\
&\quad+\textstyle\sum_{i+j=k+1}^{j>2,i>3}
	\sum_{s\in Sh(j,i-2,1)}e\left(s;f_1,\ldots,f_k\right)
		e\left(f_{s_{1}}\right)\cdots e\left(f_{s_{k-1}}\right)\\
&\quad\cdot i_{X_{D_{i-1}}\left(D_{j}\left(f_{s_{1}},\cdots,f_{s_{j}}\right),
				f_{s_{j+1}},\cdots,f_{s_{k-1}}\right)}f_{s_{k}}\\
&\quad+\textstyle\frac{1}{2}\sum_{s\in Sh(k-2,1,1)}e\left(s;f_1,\ldots,f_k\right)
	e\left(f_{s_{1}}\right)\cdots e\left(f_{s_{k-1}}\right)
		i_{X_{D_{2}}\left(D_{k-2}\left(f_{s_{1}},\ldots,f_{s_{k-2}}\right),
			f_{s_{k-1}}\right)}f_{s_{k}}
\intertext{where we already used the vanishing of equation (\ref{my_equation_3}). In
the next step we collect terms and substitute $l:=i-1$ to get}
&\textstyle\sum_{s\in S(k-1,1)}e\left(s;f_1,\ldots,f_k\right)
	e\left(f_{s_{1}}\right)\cdots e\left(f_{s_{k-1}}\right)\\
&\quad\cdot\textstyle\left(
	di_{X_{D_{k-1}}\left(f_{s_{1}},\ldots,f_{s_{k-1}}\right)}f_{s_{k}}	
		+ e\left(f_{s_{1}}\right)\cdots e\left(f_{s_{k-1}}\right)
			i_{X_{D_{k-1}}\left(f_{s_{1}},\ldots,f_{s_{k-1}}\right)}df_{s_{k}}\right)\\
&\quad+\textstyle\sum_{s\in Sh(1,k-2,1)}e\left(s;f_1,\ldots,f_k\right)
	e\left(f_{s_1}\right)\cdots e\left(f_{s_{k-1}}\right)
		i_{X_{D_{k-1}}\left(D_1f_{s_{1}},\ldots,f_{s_{k-1}}\right)}f_{s_{k}}\\
&\quad-\textstyle\sum_{s\in Sh(k-1,1)}e\left(s;f_1,\ldots,f_k\right)
	e\left(f_{s_{1}}\right)\cdots e\left(f_{s_{k-1}}\right)
		L_{X_{D_{k-1}}\left(f_{s_{1}},\cdots,f_{s_{k-1}}\right)}f_{s_{k}}\\	
&\quad+\textstyle\sum_{l+j=k}^{j>1,l>1}
	\sum_{s\in Sh(j,l-1,1)}e\left(s;f_1,\ldots,f_k\right)
		e\left(f_{s_{1}}\right)\cdots e\left(f_{s_{k-1}}\right)\\
&\quad\cdot i_{X_{D_{l}}\left(D_{j}\left(f_{s_{1}},\ldots,f_{s_{j}}\right),
				f_{s_{j+1}},\ldots,f_{s_{k-1}}\right)}f_{s_{k}}
\intertext{Now apply the definition of the Lie derivative to the first shuffle sum 
and the induction assumption together with proposition (\ref{_mvf_sh_J}) to 
the last term. Since each argument $f_j$ has the kernel property we get}
&=\textstyle\sum_{s\in S(k-1,1)}e\left(s;f_1,\ldots,f_k\right)
	e\left(f_{s_{1}}\right)\cdots e\left(f_{s_{k-1}}\right)
		L_{X_{D_{k-1}}\left(f_{s_{1}},\cdots,f_{s_{k-1}}\right)}f_{s_{k}}\\
&\quad+\textstyle\sum_{s\in Sh(1,k-2,1)}e\left(s;f_1,\ldots,f_k\right)
	e\left(f_{s_1}\right)\cdots e\left(f_{s_{k-1}}\right)
		i_{X_{D_{k-1}}\left(D_1f_{s_{1}},\ldots,f_{s_{k-1}}\right)}f_{s_{k}}\\
&\quad-\textstyle\sum_{s\in Sh(k-1,1)}e\left(s;f_1,\ldots,f_k\right)
	e\left(f_{s_{1}}\right)\cdots e\left(f_{s_{k-1}}\right)
		L_{X_{D_{k-1}}\left(f_{s_{1}},\cdots,f_{s_{k-1}}\right)}f_{s_{k}}\\	
&\quad-\textstyle\sum_{s\in Sh(1,k-2,1)}e\left(s;f_1,\ldots,f_k\right)
	e\left(f_{s_1}\right)\cdots e\left(f_{s_{k-1}}\right)
		i_{X_{D_{k-1}}\left(D_1f_{s_{1}},\ldots,f_{s_{k-1}}\right)}f_{s_{k}}
\end{align*}
and consequently the strong homotopy Jacobi equation vanishes in dimension $k$
for $k\neq 4$. 

The situation $k=4$ is just a long but straight forward computation. To see that the expression
\begin{multline*}
D_{1}D_{4}\left(f_{1},\ldots,f_{4}\right)
+ \textstyle\sum_{s\in Sh(1,3)}e\left(s;f_1,\ldots,f_4\right)
	D_{k}\left(D_{1}f_{s_{1}},f_{s_{2}},f_{s_{3}},f_{s_{4}}\right)\\
+ \textstyle\sum_{s\in Sh(3,1)}e\left(s;f_1,\ldots,f_4\right)
	D_{2}\left(D_{3}\left(f_{s_{1}},f_{s_{2}},f_{s_{3}}\right),f_{s_{4}}\right)\\
+\textstyle\sum_{s\in Sh(2,2)}e\left(s;f_1,\ldots,f_4\right)
	D_{3}\left(D_{2}\left(f_{s_{1}},f_{s_{2}}\right),f_{s_{3}},f_{s_{4}}\right)
\end{multline*}
vanishes, applying the definition of any sh-Lie $k$-bracket $D_k$ involved and 
proceed similar as in the proof of equation (\ref{my_equation_1}) to get
\begin{align*}
&\textstyle\sum_{s\in Sh(3,1)}e(s;f_1,\ldots,f_4)e(f_{s_1})e(f_{s_2})e(f_{s_3})
	di_{X_{D_3}\left(f_{s_1},f_{s_2},f_{s_3}\right)}f_{s_4}\\
&-\textstyle\sum_{s\in Sh(1,2,1)}e\left(s;f_1,\ldots,f_4\right)
	e(D_1f_{s_1})e(f_{s_2})e(f_{s_3})i_{X_{D_3}\left(D_1f_{s_1},f_{s_2},f_{s_3}\right)}f_{s_4}\\
&-\textstyle\sum_{s\in Sh(3,1)}e\left(s;f_1,\ldots,f_4\right)
	i_{X_{D_3}\left(f_{s_1},f_{s_2},f_{s_3}\right)}D_1f_{s_4}\\
&-\textstyle\sum_{s\in Sh(3,1)}e\left(s;f_1,\ldots,f_4\right)e(f_{s_1})e(f_{s_2})e(f_{s_3})
	L_{X_{D_{3}}\left(f_{s_{1}},f_{s_{2}},f_{s_{3}}\right)}f_{s_{4}}\\
&+\textstyle\sum_{s\in Sh(1,3)}e\left(s;f_1,\ldots,f_4\right)
	L_{X_{s_1}}D_{3}\left(f_{s_{2}},f_{s_{3}},f_{s_{4}}\right)\\	
&+\textstyle\frac{1}{2}\sum_{s\in Sh(2,1,1)}e\left(s;f_1,\ldots,f_4\right)e(f_{s_1})e(f_{s_2})e(f_{s_3})
	i_{X_{D_2}\left(D_{2}\left(f_{s_{1}},f_{s_{2}}\right),f_{s_{3}}\right)}f_{s_{4}}\\
&-\textstyle\frac{1}{2}\sum_{s\in Sh(2,2)}e\left(s;f_1,\ldots,f_4\right)
	i_{X_{D_2}\left(f_{s_{1}},f_{s_{2}}\right)}D_2\left(f_{s_{3}},f_{s_{4}}\right)\\
\intertext{from splitting the shuffle sum into parts that fixes te first argument and a
single term, where the first argument is suffled to the last position. Then rewrite like}
&\textstyle\sum_{s\in Sh(3,1)}e(s;f_1,\ldots,f_4)e(f_{s_1})e(f_{s_2})e(f_{s_3})\\
&\quad\cdot\left(
	di_{X_{D_3}\left(f_{s_1},f_{s_2},f_{s_3}\right)}f_{s_4}
		+e(f_{s_1})e(f_{s_2})e(f_{s_3})
			i_{X_{D_3}\left(f_{s_1},f_{s_2},f_{s_3}\right)}df_{s_4}\right)\\
&+\textstyle\sum_{s\in Sh(1,2,1)}e\left(s;f_1,\ldots,f_4\right)
	e(f_{s_1})e(f_{s_2})e(f_{s_3})i_{X_{D_3}\left(D_1f_{s_1},f_{s_2},f_{s_3}\right)}f_{s_4}\\
&-\textstyle\sum_{s\in Sh(3,1)}e\left(s;f_1,\ldots,f_4\right)e(f_{s_1})e(f_{s_2})e(f_{s_3})
	L_{X_{D_{3}}\left(f_{s_{1}},f_{s_{2}},f_{s_{3}}\right)}f_{s_{4}}\\
&-\textstyle\frac{1}{2}\sum_{s\in Sh(1,2,1)}e\left(s;f_1,\ldots,f_4\right)e(f_{s_2})e(f_{s_3})
	L_{X_{s_1}}i_{X_{D_{2}}\left(f_{s_{2}},f_{s_{3}}\right)}f_{s_{4}}\\	
&+\textstyle\frac{1}{2}\sum_{s\in Sh(2,1,1)}e\left(s;f_1,\ldots,f_4\right)e(f_{s_1})
	e(f_{s_2})e(f_{s_3})i_{X_{D_2}\left(D_{2}\left(f_{s_{1}},f_{s_{2}}\right),f_{s_{3}}\right)}
		f_{s_{4}}\\
&-\textstyle\frac{1}{2}\sum_{s\in Sh(2,1,1)}e\left(s;f_1,\ldots,f_4\right)e(f_{s_3})
	i_{X_{D_2}\left(f_{s_{1}},f_{s_{2}}\right)}L_{X_{s_{3}}}f_{s_{4}}
\intertext{and collect terms using (\ref{multi_rules}) to arrive at}
&\textstyle\sum_{s\in Sh(3,1)}e(s;f_1,\ldots,f_4)e(f_{s_1})e(f_{s_2})e(f_{s_3})
	L_{X_{D_3}\left(f_{s_1},f_{s_2},f_{s_3}\right)}f_{s_4}\\
&+\textstyle\sum_{s\in Sh(1,2,1)}e\left(s;f_1,\ldots,f_4\right)
	e(f_{s_1})e(f_{s_2})e(f_{s_3})i_{X_{D_3}\left(D_1f_{s_1},f_{s_2},f_{s_3}\right)}f_{s_4}\\
&-\textstyle\sum_{s\in Sh(3,1)}e\left(s;f_1,\ldots,f_4\right)e(f_{s_1})e(f_{s_2})e(f_{s_3})
	L_{X_{D_{3}}\left(f_{s_{1}},f_{s_{2}},f_{s_{3}}\right)}f_{s_{4}}\\
&\textstyle\frac{1}{2}\sum_{s\in Sh(2,1,1)}e\left(s;f_1,\ldots,f_4\right)e(f_{s_3})\left(
	-e(f_{s_1},f_{s_3})e(f_{s_2},f_{s_3})e(f_{s_1})e(f_{s_2})e(f_{s_3})
		L_{X_{s_3}}i_{X_{D_{2}}\left(f_{s_{1}},f_{s_{2}}\right)}f_{s_{4}}\right)\\	
&+\textstyle\frac{1}{2}\sum_{s\in Sh(2,1,1)}e\left(s;f_1,\ldots,f_4\right)
	e(f_{s_1})e(f_{s_2})e(f_{s_3})i_{X_{D_2}\left(
		D_{2}\left(f_{s_{1}},f_{s_{2}}\right),f_{s_{3}}\right)}f_{s_{4}}\\
&-\textstyle\frac{1}{2}\sum_{s\in Sh(2,1,1)}e\left(s;f_1,\ldots,f_4\right)e(f_{s_3})
	i_{X_{D_2}\left(f_{s_{1}},f_{s_{2}}\right)}L_{X_{s_{3}}}f_{s_{4}}\\
&=\textstyle\sum_{s\in Sh(1,2,1)}e\left(s;f_1,\ldots,f_4\right)
	e(f_{s_1})e(f_{s_2})e(f_{s_3})i_{X_{D_3}\left(D_1f_{s_1},f_{s_2},f_{s_3}\right)}f_{s_4}\\
&\textstyle\frac{1}{2}\sum_{s\in Sh(2,1,1)}e\left(s;f_1,\ldots,f_4\right)e(f_{s_3})
		i_{[X_{s_3},X_{D_{2}}\left(f_{s_{1}},f_{s_{2}}\right)]}f_{s_{4}}\\	
&+\textstyle\frac{1}{2}\sum_{s\in Sh(2,1,1)}e\left(s;f_1,\ldots,f_4\right)e(f_{s_1})e(f_{s_2})e(f_{s_3})
	i_{X_{D_2}\left(D_{2}\left(f_{s_{1}},f_{s_{2}}\right),f_{s_{3}}\right)}f_{s_{4}}\\
&=\textstyle\sum_{s\in Sh(1,2,1)}e\left(s;f_1,\ldots,f_4\right)
	e(f_{s_1})e(f_{s_2})e(f_{s_3})i_{
		X_{D_3}\left(D_1f_{s_1},f_{s_2},f_{s_3}\right)}f_{s_4}\\	
&+\textstyle\frac{3}{4}\sum_{s\in Sh(2,1,1)}e\left(s;f_1,\ldots,f_4\right)e(f_{s_1})
	e(f_{s_2})e(f_{s_3})i_{X_{D_2}\left(
		D_{2}\left(f_{s_{1}},f_{s_{2}}\right),f_{s_{3}}\right)}f_{s_{4}}\;.
\end{align*}
Applying the explicit expressions (\ref{ass_sHam_3}) and (\ref{well_def_2}) for the
multivector fields $X_{D_3}$ and $X_{D_2}$, using $Y_{D_1f_j}=X_{j}$, we transform this into
\begin{multline*}
\textstyle\sum_{s\in Sh(1,2,1)}e\left(s;f_1,\ldots,f_4\right)
	e(f_{s_1})e(f_{s_3})
		i_{\left[\left[X_{s_3},X_{s_2}\right],X_{s_1}\right]
		+X_{s_3}\wedge X_{D_2}\left(D_1f_{s_1},f_{s_2}\right)}f_{s_4}\\	
-\textstyle 3\sum_{s\in Sh(2,1,1)}e\left(s;f_1,\ldots,f_4\right)
	e(f_{s_1})e(f_{s_3})i_{[X_{s_{3}},[X_{s_{2}},X_{s_{1}}]]}f_{s_{4}}
\end{multline*}
The former sum over suffles vanishes due to the Jacobi identity of the Schouten
bracket and since the multivector field $X_{D_2}$ vanishes if an
argument is a closed form. The latter expression vanishes due to 
the Jacobi identity of the Schouten bracket.
This completes the proof of the strong homotopy Jacobi equation.
\section{Calculus of Differential Forms and Multivector Fields}\label{AP_2}
\subsection{Shuffle Permutation}
Let $S_k$ be the symmetric group, i.e the group of all bijective maps 
of the ordinal $\ordinal{k}$.
\begin{definition}[Shuffle Permutation] For any $p,q\in \N$
a $(p,q)$-shuffle is a permutation 
$(\mu_1,\ldots,\mu_p,\nu_1,\ldots,\nu_q)\in S_{p+q}$
subject to the condition $\mu_1<\ldots<\mu_p$ and $\nu_1<\ldots<\nu_p$.
We write $Sh(p,q)$ for the set of all $(p,q)$-shuffles.
\end{definition}
For more on shuffles, see for example at \cite{RS}.
\subsection{Graded Vector Spaces}We recall the most basics facts
about graded vector spaces. 

A $\Z$-graded $\mathbb{K}$-vector space $V$ is the direct sum 
$\oplus_{n\in \Z} V_n$ of $\mathbb{K}$-vector spaces $V_n$. The elements of $V_n$ are said to be \textbf{homogeneous} of degree $n$. Obviously every vector has a decomposition 
into homogeneous elements. When the degree of a vector $v \in V$  is well defined, i.e.
when the vector is homogeneous we denote it by $deg(v)$ (or by $|v|$ if we have are
dealing with more than one grading). In what follows we assume all vector space to be defined over $\R$ and consequently we just write vector space instead of $\R$-vector 
space.

A morphism $f : V \to W$ of graded vector spaces is a sequence of 
linear maps $f_n : V_n \to W_{n+r}$ for all $n\in \Z$. 
The integer $r$ is called the degree of $f$ and is as well denoted by $deg(f)$ (or $|f|$). 

A $k$-linear morphism $f: V_1 \times ... \times V_k \to W$ 
of graded vector spaces is a sequence of $k$-linear maps 
$f_{n_1,\ldots,n_k} : (V_1)_{n_1} \times \ldots \times (V_k)_{n_k} \to W_{\sum n_i+r}$ for all $n_i\in \Z$.

The (graded )tensor product $V \otimes W$ of two graded $\mathbb{K}$-vector spaces
$V$ and $W$ is given by
$$
(V \otimes W )_n :=\oplus_{i+j=n}\left( V_i \otimes W_j\right)
$$
and the twisting morphism by
$\tau: V \otimes W \to  W \otimes V$ on homogeneous elements $v\otimes w \in V \otimes W$
by 
$$\tau(v \otimes w):=(-1)^{deg(v)deg(w)} w \otimes v$$ and then extended to
$V\otimes W$ by linearity.

According to a better readable text we define $e(v):=(-1)^{deg(v)}$ as 
well as $e(v,w):=(-1)^{deg(v)deg(w)}$. For any permutation $s\in S_k$ and
any homogeneous vectors $v_1,\ldots,v_k\in V$ we define the \textbf{Koszul sign}
$e(s;v_1,\ldots,v_k) \in \{-1,+1\}$ by 
\begin{equation}
v_1\otimes \ldots \otimes v_k= e(s;v_1,\ldots,v_k) v_{s_1}\otimes \ldots \otimes v_{s_k}.
\end{equation}
\begin{remark}
In an actual computation the Koszul sign can be determined by the following rules: 
If a permutation $s\in S_k$ interchanges $j$ and $j+1$, 
then $e(s;v_1,\ldots,v_k)= (-1)^{deg(v_j)\cdot deg(v_{+1})}$. 
If $t\in S_k$ is another permutation, then
$e(ts;v_1,\ldots,v_k)=e(t;v_{s_1},\ldots,v_{s_k})e(s;v_1,\ldots,v_k)$.
\end{remark}
A $k$-linear morphism $ f: \bigtimes^k V \to W$ is called 
\textbf{graded symmetric} if 
$$f(v_1,\ldots,v_k) = e(s;v_1,\ldots,v_k)f(v_{s(1)},\ldots,v_{s(k)})$$
for all $s\in S_k$.

\subsection{Calculus on Multivector Fields} For a comprehensive definition of
multivector fields and the Schouten bracket see for example \cite{MM} or \cite{KMS}.

Let $M$ be a smooth manifold.
A multivector field $X$ of tensor degree $r$ is a section 
of the $r$-th exterior power $\bwedge{r}TM$ of the tangent bundle. 
We write $\mathfrak{X}(M)$ for the set of all multivector fields. 

The \textbf{Schouten bracket} is a graded antisymmetric, 
natural $\mathbb{R}$-bilinear operator
(in the sense of \cite{KMS})
\begin{equation}
[.,.]:\mathfrak{X}M \times \mathfrak{X}M \to \mathfrak{X}M\; ,
\end{equation}
homogeneous of tensor degree $-1$, that satisfies the graded Leibniz rule 
\begin{equation}
[X,Y \smwedge Z] = [X,Y]\smwedge Z + (-1)^{(|\,X\,|-1)|\,Y\,|} Y \smwedge [X,Z]\;,
\end{equation}
as well as the graded Jacobi identity
\begin{equation}
\textstyle\sum_{s\in Sh(2,1)}e(s;X_1,X_2,X_3)[[X_{s_1},X_{s_2}],X_{s_3}]=0\;.
\end{equation}
Moreover it coincides with the standard Lie bracket on vector fields.

If $\alpha \in \Omega(M)$ is a differential form and $X\in \mathfrak{X}M$ a 
multivector field of tensor degree $r$, the \textbf{contraction} $i_X\alpha$
of $\alpha$ along $X$, is defined for decomposable multivector fields
$X_1\smwedge \ldots \smwedge X_r $ by repeated contraction
\begin{equation}\label{contract}
 i_{X_1 \wedge \ldots \wedge X_r} \alpha
 =i_{X_r} \ldots i_{X_1} \alpha
\end{equation}
and is then extended to arbitrary multivector fields $X$ by linearity. The 
\textbf{Lie derivative} $L_X \alpha$ along $X$ is defined in analogy to the Cartan formula for vector fields, 
as the graded commutator of the exterior derivative
$d$ and the contraction operator $i_X$ according to:
\begin{equation}\label{def_lie}
  L_X \alpha = d\left(i_X\alpha\right) - (-1)^r i_X d \alpha.
\end{equation}

\begin{prop}\label{multi_rules}
For multivector fields $X$ of tensor degree $r$ and $Y$ of tensor degree $s$,
the equations 
\begin{equation}
\begin{array}{lcl}
dL_X^{} \alpha &=&(-1)^{r-1}L_X d\alpha \\
i_{[X,Y]} \alpha &=&(-1)^{(r-1)s}L_X i_Y^{} \alpha-i_Y L_X \alpha\\
L_{[X,Y]} \alpha &=& (-1)^{(r-1)(s-1)}L_X L_Y \alpha-L_Y L_X \alpha\\
L_{X\wedge Y} \alpha &=& (-1)^{s} i_Y L_X \alpha + L_Y i_X \alpha 
\end{array}
\end{equation}
are satisfied for any $\alpha\in \Omega(M)$.
\end{prop}
\begin{proof}
See for example \cite{FPR}.
\end{proof}
\begin{definition}
Let $f : M \to N$ be a smooth map. Two r-multivector fields 
$X\in \mathfrak{X}M$ and $Y\in \mathfrak{X}N$ are called 
\textbf{f-related}, if 
$$\wedge^r\, Tf\circ X = Y\circ f\;.$$
\end{definition}
If $f:M \to N$ is a diffeomorpism and $f^*Y$ is the pullback of a multivector field
$Y\in \mathfrak{X}N$, then $f^*Y$ and $Y$ are $f$-related.
\begin{prop}
If $f : M \to N$ be a smooth map and $X\in \mathfrak{X}M$ and $Y\in \mathfrak{X}N$ are
$f$-related multivector fields, then $i_X \circ f^* = f^* \circ i_Y$.
\end{prop}
\begin{definition}[Kernel] Let $\alpha$ be differential form on a manifold $M$.
The \textbf{kernel} of $\alpha$ is the set
\begin{equation}
\ker(\alpha):=\left\{X \in \mathfrak{X}\left(M\right)\; |\; 
i_X\alpha = 0_{C^\I(M)}\right\}\;.
\end{equation}
of all multivector fields $X$ on $M$, such that the contraction of $\alpha$ along
$X$ vanishes.
\end{definition}
\end{appendix}

\end{document}